\documentclass[final,leqno,onefignum,onetabnum]{siamltex1213}
\usepackage{bbm}
\usepackage{amsmath,amsfonts,amssymb,enumerate}
\usepackage{color}
\usepackage{lscape}
\usepackage{latexsym}
\usepackage{multirow}
\usepackage{rotating}
\usepackage{threeparttable}
\usepackage{graphicx}
\usepackage{algpseudocode,algorithm}
\usepackage[margin=1in]{geometry}
\newtheorem{example}[theorem]{Example}
\newtheorem{remark}[theorem]{Remark}
\newtheorem{hypothesis}{H}
\graphicspath{{figs/}}
\definecolor{doushalv}{rgb}{0.78,0.93,0.80}

\title{Capped $\ell_p$ approximations for the composite $\ell_0$ regularization problem \thanks{This research is supported in part by Guangdong Provincial Government of China through the ``Computational Science Innovative Research Team'' program, by the Natural Science Foundation of China under grants 11501584 and 11626103, by the
Natural Science Foundation of Guangdong Province under grants 2014A030310332 and 2014A030310414, and by the Fundamental Research Funds for the Central Universities of China.}}

\author{Qia Li \thanks{Guangdong Province Key Laboratory of Computational Science, School of Data and Computer Sciences, Sun Yat-sen University, Guangzhou 510275, P. R. China.}
\and Na Zhang \thanks{Department of Applied Mathematics, College of Mathematics and Informatics, South China Agricultural University, Guangzhou 510642, P. R. China (nzhsysu@gmail.com). Questions, comments, or corrections to this document may be directed to that email address.}
}

\begin{document}
\maketitle
\begin{abstract}
The composite $\ell_0$ function serves as a sparse regularizer in many applications. The algorithmic difficulty caused by the composite $\ell_0$ regularization (the $\ell_0$ norm composed with a linear mapping) is usually bypassed through approximating the $\ell_0$ norm. We consider in this paper capped $\ell_p$ approximations with $p>0$ for the composite $\ell_0$ regularization problem. For each $p>0$, the capped $\ell_p$ function converges to the $\ell_0$ norm pointwisely as the approximation parameter tends to infinity.
We point out that the capped $\ell_p$ approximation problem is essentially a  penalty  method with an $\ell_p$ penalty function for the composite $\ell_0$ problem from the viewpoint of numerical optimization. Our theoretical results stated below may shed a new light on the penalty  methods for solving the composite $\ell_0$ problem and help the design of innovative numerical algorithms. We first establish the existence of optimal solutions to the composite  $\ell_0$ regularization problem and its capped $\ell_p$ approximation problem under conditions that the data fitting function is asymptotically level stable and bounded below. Asymptotically level stable functions cover a rich class of data fitting functions encountered in practice. We then prove that the capped $\ell_p$ problem asymptotically approximates the composite $\ell_0$ problem if the data fitting function is a level bounded function composed with a linear mapping. We further show that if the data fitting function is the indicator function on an asymptotically linear set or the $\ell_0$ norm composed with an affine mapping, then the composite $\ell_0$ problem and its capped $\ell_p$ approximation problem share the same optimal solution set provided that the approximation parameter is large enough.
\end{abstract}

\section{Introduction}
Structured sparsity regularization has been  successfully  applied to ill-conditioned inverse problems in the area of image processing, machine learning and statistics. For example, in image processing, the underlying image always becomes sparse by a properly chosen transform. Over the past decade, the $\ell_1$ norm is widely utilized to measure the sparsity. Let the proper, lower semicontinuous and bounded below function $\phi:\mathbb{R}^n\rightarrow \bar{\mathbb{R}}:=\mathbb{R}\cup\{+\infty\}$ stand for the data fitting term. Then an extensively used approach is to solve the following composite $\ell_1$ regularization problem
\begin{equation}\label{model:l1}
\min\{\phi(x)+\lambda\|Bx\|_1: x\in\mathbb{R}^n\},
\end{equation}
where $B$ is an $m\times n$ real matrix and $\lambda>0$ is a regularization parameter.
In problem \eqref{model:l1}, the composite $\ell_1$ regularizer  $\|Bx\|_1$ is applied to promoting the sparsity of the vector $Bx$. Problem \eqref{model:l1} is also known as the $\ell_1$ analysis based approach proposed in \cite{Elad-Milanfar-subinstein:IP2007,Elad:ACHA:05}. Since the $\ell_0$ norm of a vector counts the number of its nonzero entries, it is more natural to make use of the composite $\ell_0$ regularizer $\|Bx\|_0$. This leads to the following composite $\ell_0$ regularization problem
\begin{equation}\label{model:l0}
\min\{\Phi(x):=\phi(x)+\lambda\|Bx\|_0: x\in\mathbb{R}^n\}.
\end{equation}
In image restoration, it is demonstrated in \cite{DongJSC, Portilla:ConferenceIP2009, Shen-xu-zeng:acha2016, Trzasko-Mandoca:IEEEIP2009, Zhang-Dong-Lu:MathCompt:2013} that problem \eqref{model:l0} generates images with better quality than those obtained by problem \eqref{model:l1}.

The composite $\ell_0$ norm imposes computational difficulties on solving problem \eqref{model:l0}. First, finding a global minimizer of \eqref{model:l0} is known to be NP-hard in general \cite{Davis1997Adaptive,Natarajan1995Sparse,Tropp:2006} due to the $\ell_0$ norm. Moreover, although some algorithms such as greedy matching pursuit methods \cite{Needell-Tropp:2009, Tropp-Gilbert:2007} and iterative hard thresholding algorithms \cite{Blumensath-Davies2009} are very popular and efficient for the $\ell_0$ norm minimization, they can only be  applied to the non-composite $\ell_0$ regularization problem, i.e., the case when $B$ is an identity matrix. Therefore, approximations of $\ell_0$ norm are frequently used in numerical algorithms for problem \eqref{model:l0}. Many nonconvex  sparsity regularization functions may be adopted to approximate the $\ell_0$ norm \cite{Candes-Wakin-Boyd:JFAA:08,Chartrand:IP2008, Fan2002Variable,Foucart-lai:ACHA2009,Woodworth:IP2016,Zhang:ANNALS:2010,Zhang:JMLR:2010}.
The alternative approximate problems always have better structures from the viewpoint of algorithmic design. For example,  the majorization-minimization strategy can be applied to develop efficient algorithms for solving the approximate problems, see \cite{Candes-Wakin-Boyd:JFAA:08, Chouzenoux-Jezierska-Pesquet-Talbot:2013, Lu:2014} for instance.

In this paper, we consider using capped $\ell_p$ functions with $p>0$ to  approximate the $\ell_0$ function.
For $p>0$, the capped $\ell_p$ function  $\psi_\gamma: \mathbb{R}^m\rightarrow\mathbb{R}$ with $\gamma>0$ at $y\in\mathbb{R}^m$ is defined by
\begin{equation}\label{def:psi}
\psi_\gamma(y)=\sum_{i=1}^m \varphi_{\gamma}(y_i),
\end{equation}
where $\varphi_\gamma(y_i)=\min(\gamma|y_i|^p,1)$.
In fact, the scalar capped $\ell_p$ function $\varphi_\gamma$ is a piecewise function as follows
\begin{equation}\label{def:varphi}
\varphi_\gamma(t)=\begin{cases}
1, & |t|\ge \frac{1}{\gamma^{1/p}},\\
\gamma |t|^p, &\mathrm{else}.
\end{cases}
\end{equation}
We exhibit the capped $\ell_p$ function $\varphi_\gamma$ in Figure \ref{fig:varphi}.
\begin{figure}
\centering
\scalebox{0.5}{\includegraphics{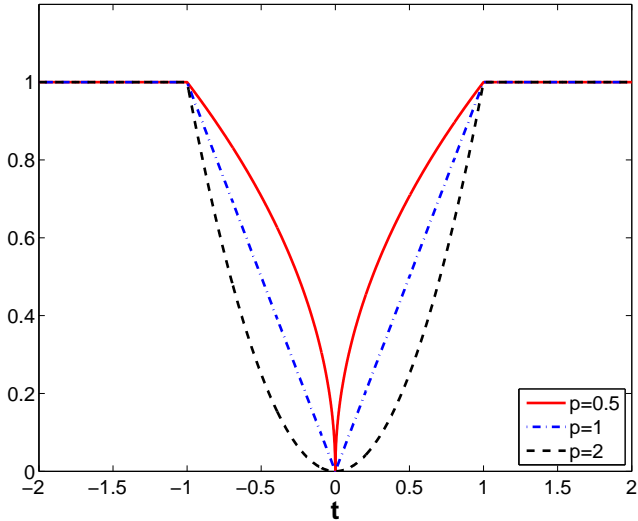}}
\caption{Capped $\ell_p$ function $\varphi_\gamma$ with $\gamma=1$ for $p=0.5, 1, 2$. }
\label{fig:varphi}
\end{figure}

The capped $\ell_1$ function \cite{Zhang:JMLR:2010}, capped $\ell_2$  function \footnote{The capped $\ell_2$ function is also referred to as   the truncated quadratic function in some literatures.}\cite{Chouzenoux-Jezierska-Pesquet-Talbot:2013}, capped $\ell_p$ functions with $0<p<1$ \cite{Nikolova-Ng-Tam:2010} and with $p\geq 1$ \cite{Fornasier-ward:foundations2010} have been successfully applied to promoting sparsity.
It is clear that $\psi_{\gamma}(y)\rightarrow \|y\|_0$ as $\gamma\rightarrow +\infty$ for any $y\in\mathbb{R}^m$, which means that $\{\psi_\gamma:\gamma>0\}$ asymptotically approximates the $\ell_0$ norm with respect to $\gamma$. By employing capped $\ell_p$ functions, the resulting approximate problem of problem \eqref{model:l0} is then given by
\begin{equation}\label{model:3}
\min\{\Psi_\gamma(x):=\phi(x)+\lambda\psi_\gamma(Bx): x\in\mathbb{R}^n\}.
\end{equation}
We remark that $\psi_\gamma$ at $y\in\mathbb{R}^m$ can be written in an equivalent form (see Appendix)
\begin{equation}\label{eq:split}
\psi_\gamma(y)=\min\{\|v\|_0+{\gamma}\|y-v\|_p^p: v\in\mathbb{R}^m\},
\end{equation}
where $\|\cdot\|_p$ with $p>0$ at $z\in\mathbb{R}^m$ is defined by $\|z\|_p=(\sum_{i=1}^m |z_i|^p)^{1/p}$.
Therefore, problem \eqref{model:3} is completely consistent with the following two variables optimization problem
\begin{equation}\label{eq:6.3.1}
\min\{\phi(x)+\lambda \gamma\|Bx-v\|_p^p+\lambda\|v\|_0: x\in\mathbb{R}^n, v\in\mathbb{R}^m\}.
\end{equation}
More precisely, if $x^*$ is a global minimizer of problem \eqref{model:3}, then there exists a $v^*\in\mathbb{R}^m$ such that $(x^*, v^*)$ is a global minimizer of problem \eqref{eq:6.3.1}.
Conversely, if the pair $(x^*, v^*)$ is a global minimizer of problem \eqref{eq:6.3.1}, then $x^*$ is a global minimizer of problem \eqref{model:3}.
It is obvious that \eqref{model:l0} can be equivalently reformulated as
\begin{equation}\label{eq:0.1}
\min\{\phi(x)+\lambda\|v\|_0: Bx=v, x\in\mathbb{R}^n, v\in\mathbb{R}^m\}.
\end{equation}
The formulations of \eqref{eq:6.3.1} and \eqref{eq:0.1} facilitate us to understand the capped $\ell_p$ approximations to the composite $\ell_0$ norm from the standpoint of numerical optimization methods. Actually,
problem \eqref{model:3} (i.e., problem \eqref{eq:6.3.1}) is essentially a  penalty  method  with the $\ell_p$ penalty function for solving problem  \eqref{model:l0} (i.e., problem \eqref{eq:0.1}). Numerical schemes such as nonconvex block coordinate decent algorithms \cite{Attouch-bolt-redont-soubeyran:2010,attouch-Convergence,Grippof-Sciandrone:2010} can be easily and efficiently adapted for solving problem \eqref{eq:6.3.1} especially in the case of $p=2$.

As usual,  we use ``optimal solutions'' for global minimizers and ``optimal solution set'' for the set of all global minimizers. To the best of our knowledge, there is little theory concerning optimal solutions to problems \eqref{model:l0} and \eqref{model:3} generally. However, some theoretical findings in the literature are related to this topic. In \cite{Fornasier-ward:foundations2010}, optimal solutions to problems \eqref{model:l0} and \eqref{model:3} are investigated in the special case where $B$ is identity, $p\geq 1$ and $\phi$ is  quadratic. If $\phi$ satisfies conditions like the restricted isometry property \cite{candes2008RIP} and $\lambda$ is larger than a threshold, it is shown in \cite{Fornasier-ward:foundations2010} that optimal solutions to problem \eqref{model:3} can be asymptotically obtained by problem \eqref{model:l0}. This asymptotic convergence results may be derived for arbitrary $B\in\mathbb{R}^{m \times n}$ and $p=2$ under the condition that $\phi$ is level bounded, by generalizing the analysis of \cite{Chouzenoux-Jezierska-Pesquet-Talbot:2013}, in which the authors employ the $\ell_0-\ell_2$ regularizer to approximate the $\ell_0$ norm.
However, in practice the conditions required by \cite{Fornasier-ward:foundations2010} or \cite{Chouzenoux-Jezierska-Pesquet-Talbot:2013} are usually not satisfied. It will be much better if we can establish theoretical results under mild conditions which are generally satisfied in various applications.

Our main contributions in this paper are summarized below. We expect they can shed a new light on penalty methods for solving problem \eqref{model:l0} and  give rise to innovative numerical schemes.
\begin{itemize}
\item We establish the existence of optimal solutions to problems \eqref{model:l0} and \eqref{model:3} under the conditions that $\phi$ is asymptotically level stable and bounded below. As it will be shown in Section \ref{sec:als}, the function $\phi$ is always asymptotically level stable in a wide range of applications.

\item We obtain that if $\phi$ is a level bounded function composed with a linear mapping, then problem \eqref{model:3} asymptotically approximates problem \eqref{model:l0} as $\gamma$ goes to infinity. Let $\{\gamma^k>0: k\in\mathbb{N}\}$ be an increasing sequence going to infinity and $x^k\in\arg\min\Psi_{\gamma^k}$. Then, $\min\Psi_{\gamma^k}\rightarrow\min\Phi$ as $k\rightarrow+\infty$ and any cluster point of $\{x^k: k\in\mathbb{N}\}$ is an optimal solution to problem \eqref{model:l0}. We emphasis here that we establish the asymptotic convergence results under conditions that $\phi$ is a level bounded function composed with a linear mapping, which covers a more general class of data fitting functions than those in \cite{Fornasier-ward:foundations2010} and \cite{Chouzenoux-Jezierska-Pesquet-Talbot:2013}.
\item We provide two cases where problem \eqref{model:3} is an exact approximation to problem \eqref{model:l0} when $\gamma$ is sufficiently large. More precisely, we show that if $\phi$ is the indicator function on an asymptotically linear set or
    the $\ell_0$ function composed with an affine mapping, then there exists a $\gamma^*>0$ such that both problems share the same optimal solution set for $\gamma>\gamma^*$.
\end{itemize}

The remaining part of this paper is organized as follows. In Section 2 we provide some preliminary results on asymptotically level stable functions. The existence of optimal solutions to problems \eqref{model:l0} and \eqref{model:3} are investigated in Section 3.
Section 4 establishes that problem \eqref{model:3} asymptotically approximates problem \eqref{model:l0}. Stability for problem \eqref{model:l0} are discussed in Section 5 and two cases are provided in Section 6 where problem \eqref{model:3} is an exact approximation to problem \eqref{model:l0} when $\gamma$ is large enough. We conclude this paper in Section 7.

\section{Asymptotically level stable functions}\label{sec:als}
In this section, we provide some preliminary results on asymptotically level stable functions. We show that in various applications the data fitting functions $\phi$ are usually asymptotically level stable. The notion of asymptotically level stable functions serves as a basis for establishing existence of optimal solutions to optimization problems in Section 3 and Section 5.

We first recall the notion  of asymptotically level stable functions. For concise presentation, we denote by $\mathbb{N}$ all the positive integers, that is $\mathbb{N}:=\{1,2,\dots\}$. Further, for any positive integer $k$ we define
$$
\mathbb{N}_k:=\{1,2,\dots, k\} \mathrm{~~and~~} \mathbb{N}_k^0:=\{0,1,\dots, k\}.
$$
Denote by $\mathrm{lev}(f, \alpha)$ the level set of $f:\mathbb{R}^n\rightarrow \bar{\mathbb{R}}$ at $\alpha\in\mathbb{R}$, that is $\mathrm{lev}(f, \alpha):=\{x\in\mathbb{R}^n: f(x)\le \alpha\}$.
The following definition is taken from 
\cite{Teboulle:asympt2003}.
\begin{definition}\label{def:als}
Let $f:\mathbb{R}^n\rightarrow\bar{\mathbb{R}}$ be lower semicontinuous and proper. Then $f$ is said to be asymptotically level stable if for each $\rho>0$, each bounded sequence of
reals $\{t^k: k\in\mathbb{N}\}$ and each sequence
$\{x^k: k\in\mathbb{N}\}$ satisfying
\begin{equation}\label{eq:alsCondition}
x^k\in\mathrm{lev}(f, t^k), \|x^k\|_2\rightarrow+\infty, x^k\|x^k\|_2^{-1}\rightarrow \bar{x}\in\mathrm{ker}(f_\infty),
\end{equation}
where $f_\infty$ denotes the asymptotic function (one can refer to Definition 2.5.1 of \cite{Teboulle:asympt2003}) of $f$, there exists $k_0\in\mathbb{N}$ such that
$$
x^k-\rho\bar{x}\in\mathrm{lev}(f, t^k)
$$
for any $k\ge k_0$.
\end{definition}

 A trivial case of asymptotically level stable functions is that for each bounded sequence of reals $\{t^k: k\in\mathbb{N}\}$, there exists no sequence $\{x^k: k\in\mathbb{N}\}$ satisfying \eqref{eq:alsCondition}.
Recall that a proper function $f:\mathbb{R}^n\rightarrow \bar{\mathbb{R}}$ is called level bounded if $\lim_{\|x\|_2\rightarrow +\infty}f(x)=+\infty$. Obviously, level bounded functions are asymptotically level stable. However, in many applications the loss functions are level bounded functions composed with a linear mapping, which are not necessarily level bounded.

We present a proposition regarding the composition of a level bounded function and a linear mapping.
\begin{proposition}\label{prop:leveBoun}
Let $A\in\mathbb{R}^{t\times n}$ and $g:\mathbb{R}^t\rightarrow \bar{\mathbb{R}}$ be proper, lower semicontinuous and level bounded.  Then $f:=g\circ A$ is asymptotically level stable.
\end{proposition}
\begin{proof}
Clearly, $f$ is proper and lower semicontinuous. It is obvious that when $\mathrm{rank}(A)=n$,  $f$ is level bounded due to the level boundedness of $g$.   Therefore, $f$ is asymptotically level stable since there does not exist any sequence $\{x^k: k\in\mathbb{N}\}$ satisfying the conditions in Definition \ref{def:als}.

We next study the case when $\mathrm{rank}(A)<n$. Let $\rho>0$ and $\{t^k: k\in\mathbb{N}\}$ be any bounded sequence of reals. Let $\{x^k: k\in\mathbb{N}\}$ be any sequence satisfying \eqref{eq:alsCondition}. By Definition \ref{def:als}, if  $\bar{x}\in\mathrm{ker}(A)$ then $f$ is asymptotically level stable.  Thus, we dedicate to proving $\bar{x}\in\mathrm{ker}(A)$ following.

For any $k\in\mathbb{N}$, the vector $x^k$ has the unique  decomposition $x^k=u^k+w^k$, where $u^k\in\mathrm{ker}(A)^\perp$ and $w^k\in\mathrm{ker}(A)$. Since $\{t^k: k\in\mathbb{N}\}$ is bounded,  $\{Ax^k: k\in\mathbb{N}\}$ is bounded due to the level boundedness of $g$. Thus, $\{u^k: k\in \mathbb{N}\}$ is bounded. Then, $\lim_{k\rightarrow \infty}\frac{w^k}{\|x^k\|_2}=\lim_{k\rightarrow \infty}\frac{x^k}{\|x^k\|_2}=\bar{x}$. Therefore,  $\bar{x}\in\mathrm{ker}(A)$ since
$\frac{w^k}{\|x^k\|_2}\in\mathrm{ker}(A)$.
Then, we complete the proof.
\end{proof}

In applications when the  noise obeys the two-point distribution or the multipoint distribution, it is very effective to involve the $\ell_0$ norm in the data fitting function. Clearly, the $\ell_0$ norm is not a level bounded function. We next prove that the $\ell_0$ function composed with an affine mapping is also asymptotically level stable.
\begin{proposition}\label{prop:l0ALS}
Let $A\in\mathbb{R}^{t\times n}$ and $b\in\mathbb{R}^t$. Then $f:=\|A\cdot-b\|_0$ is asymptotically level stable.
\end{proposition}
\begin{proof}
One can check that $f$ is proper and lower semicontinuous. Let $\rho>0$ and $\{t^k: k\in\mathbb{N}\}$ be any real bounded sequence. Let $\{x^k: k\in\mathbb{N}\}$ be any sequence satisfying \eqref{eq:alsCondition}. Set $\Lambda:=\mathrm{supp}(A\bar{x})$ and $\Lambda^C:=\mathbb{N}_t\backslash \Lambda$. Then for $i\in\Lambda^C$, $(A(x^k-\rho\bar{x})-b)_i=(Ax^k-b)_i$.
For $i\in\Lambda$, $(Ax^k)_i\rightarrow +\infty$ or $(Ax^k)_i\rightarrow -\infty$  since $\|x^k\|_2\rightarrow +\infty$ and $x^k\|x^k\|_2^{-1}\rightarrow \bar{x}$. Thus, there exists $k_0>0$ such that $(Ax^k)_i-b_i\neq 0$ for all $k\ge k_0$ and all $i\in\Lambda$. Then, there holds $\|(A(x^k-\rho\bar{x})-b)_i\|_0\le\|(Ax^k-b)_i\|_0=1$ for all $k\ge k_0$ and all $i\in\Lambda$. Therefore, we obtain that
\begin{eqnarray}
\|A(x^k-\rho\bar{x})-b\|_0&=&\sum_{i\in\Lambda}\|(A(x^k-\rho\bar{x})-b)_i\|_0+\sum_{i\in\Lambda^C}\|(A(x^k-\rho\bar{x})-b)_i\|_0\\
&\le &\sum_{i\in\Lambda}\|(Ax^k-b)_i\|_0+\sum_{i\in\Lambda^C}\|(Ax^k-b)_i\|_0\\
&=&\|Ax^k-b\|_0\\
&\le & t^k
\end{eqnarray}
holds for any $k\ge k_0$. This proposition follows immediately.
\end{proof}

We further  prove in the next proposition that  the sum of an asymptotically level stable function and the $\ell_0$ function is also asymptotically level stable. We require to study the asymptotic function. Recall that, for a proper function $f:\mathbb{R}^n\rightarrow \bar{\mathbb{R}}$, a functional analytic representation of the asymptotic function $f_\infty$ defined at $x\in\mathbb{R}^n$ is given by (Theorem 2.5.1 of \cite{Teboulle:asympt2003})
\begin{equation}\label{eq:AsyFunAna}
f_{\infty}(x)=\mathrm{liminf}_{x'\rightarrow x, t\rightarrow +\infty} \frac{f(tx')}{t}.
\end{equation}
With the help of the above analytic representation, we obtain the following lemma.

\begin{lemma}\label{lema:PhiInfty}
Let $g:\mathbb{R}^n\rightarrow \bar{\mathbb{R}}$ be proper and $h:\mathbb{R}^n\rightarrow \mathbb{R}$ be a bounded function.  Let $f:=g+h$. Then
 $f_\infty=g_\infty.$
\end{lemma}
\begin{proof}
It is clear that $f$ is proper since $g$ is proper and $h$ is bounded. Then  $f$ has the asymptotic function.
According to \eqref{eq:AsyFunAna},  for any $x\in\mathbb{R}^n$,
$$
\begin{array}{rcl}
f_\infty(x)&=&\liminf_{x'\rightarrow x, t\rightarrow +\infty} \frac{f(tx')}{t}\\
&=&\liminf_{x'\rightarrow x, t\rightarrow +\infty} \frac{g(tx')+h(tx')}{t}\\
&=&\liminf_{x'\rightarrow x, t\rightarrow +\infty} \frac{g(tx')}{t}\\
&=&g_\infty (x).
\end{array}
$$
We then get this lemma.
\end{proof}
\begin{proposition}\label{prop:Asy+L0asl}
Let $A\in\mathbb{R}^{t\times n}$, $b\in\mathbb{R}^t$ and $\lambda>0$.  If $g:\mathbb{R}^n\rightarrow\bar{\mathbb{R}}$ is asymptotically level stable, then $f:=g+\lambda\|A\cdot-b\|_0$  is also asymptotically level stable.
\end{proposition}
\begin{proof}
It is obvious that $f$ is  lower semicontinuous and proper since $g$ is lower semicontinuous and proper.
Let $\rho>0$ and $\{t^k: k\in\mathbb{N}\}$ be any bounded sequence of reals. Let  $\{x^k: k\in\mathbb{N}\}$ be any sequence   satisfying
\eqref{eq:alsCondition}.
Our task is  proving that there exists $k_0>0$ such that
\begin{equation}\label{eq:eq2..3}
f(x^k-\rho\bar{x})=g(x^k-\rho\bar{x})+\lambda\|A(x^k-\rho\bar{x})-b\|_0\le t^k
\end{equation}
 for any $k\ge k_0$.

By Lemma \ref{lema:PhiInfty}, $\mathrm{ker}(f_\infty)=\mathrm{ker}(g_\infty)$. Set $\tau^k:=t^k-\lambda\|Ax^k-b\|_0$. Then $\{\tau^k: k\in\mathbb{N}\}$ is bounded. Thus, \eqref{eq:alsCondition} implies
$$
x^k\in\mathrm{lev}(g, \tau^k), \|x^k\|_2\rightarrow +\infty, x^k\|x^k\|_2^{-1}\rightarrow \bar{x}\in\mathrm{ker}(g_\infty).
$$
Since $g$ is asymptotically level stable, we have that there exists $k_1>0$ such that
\begin{equation}\label{eq:0.0}
g(x^k-\rho\bar{x})\le \tau^k=t^k-\lambda\|Ax^k-b\|_0
\end{equation}
 for any $k\ge k_1$.

Finally, we dedicate to showing that there holds $k_0>0$ such that
\begin{equation}\label{eq:2.3}
\|A(x^k-\rho\bar{x})-b\|_0\le \|Ax^k-b\|_0
\end{equation}
 for any $k\ge k_0$.
Set $\Lambda:=\mathrm{supp} (A\bar{x})$. Obviously,  $(A(x^k-\rho\bar{x})-b)_i=(Ax^k-b)_i$ for all $i\in\Lambda^C$. For $i\in\Lambda$, $(Ax^k)_i\rightarrow \infty$ due to $\|x^k\|_2\rightarrow +\infty$ and $x^k\|x^k\|_2^{-1}\rightarrow \bar{x}$. Thus, there exists $k_0>k_1$ such that $(Ax^k-b)_i\neq 0$, therefore,  $\|(A(x^k-\rho\bar{x})-b)_i\|_0\le \|(Ax^k-b)_i\|_0=1$ for any $i\in\Lambda$ and any $k\ge k_0$. It follows that
$$
\begin{array}{rcl}
\|A(x^k-\rho\bar x)-b\|_0&=&\sum_{i\in\Lambda} \|(A(x^k-\rho\bar{x})-b)_i\|_0+\sum_{i\in\Lambda^C}\|(A(x^k-\rho\bar{x})-b)_i\|_0\\
&\le&\sum_{i\in\Lambda} \|(Ax^k-b)_i\|_0+\sum_{i\in\Lambda^C}\|(Ax^k-b)_i\|_0\\
&=& \|Ax^k-b\|_0
\end{array}
$$
holds for any $k\ge k_0$.

Then inequality \eqref{eq:2.3} and \eqref{eq:0.0} together imply \eqref{eq:eq2..3}. We then complete the proof.
\end{proof}

An interesting corollary based on the previous propositions is presented below.
\begin{corollary}\label{coro:als}
Let $T\in\mathbb{N}$, $A_i\in\mathbb{R}^{t_i\times n}$ and $b_i\in \mathbb{R}^{t_i}$, $i\in\mathbb{N}_{T+1}$.  Then the following statements hold:
\begin{itemize}
\item [(i)] $f_1:=\|A_1\cdot-b_1\|_q^q$ is asymptotically level stable for $q>0$.
\item [(ii)] $f_2:=\sum_{i=1}^T \|A_i\cdot-b_i\|_{q_i}^{q_i}$ is asymptotically level stable for $q_i>0$, $i\in\mathbb{N}_T$.
\item [(iii)] $f_3:=\sum_{i=1}^{T} \|A_i\cdot-b_i\|_{q_i}^{q_i}+\lambda\|A_{T+1} \cdot-b_{T+1}\|_0$ is asymptotically level stable for $q_i>0$, $i\in\mathbb{N}_T$, where $\lambda>0$.
\end{itemize}
\end{corollary}
\begin{proof}
Item (i) follows from Proposition \ref{prop:leveBoun} and the fact that $\|\cdot-b_1\|_q^q$ is level bounded  for $q>0$. In order to prove Item (ii), we set $A:=[A_1; A_2; \dots; A_T]$, $b:=[b_1;b_2;\dots;b_T]$
 and $g:\mathbb{R}^{\sum_{i=1}^T t_i}\rightarrow\mathbb{R}$ defined at $(y_1, y_2,\dots, y_T)\in\mathbb{R}^{t_1}\times\mathbb{R}^{t_2\times}\dots\times\mathbb{R}^{t_T}$
 as $g(y_1,y_2,\dots,y_T):=\sum_{i=1}^T \|y_i-b_i\|_{q_i}^{q_i}$. Obviously, $g$ is level bounded  and $f_2=g\circ A$. According to Proposition \ref{prop:leveBoun}, Item (ii) is obtained  immediately. Item (iii) is a direct result of Item (ii) and Proposition \ref{prop:Asy+L0asl}.
 \end{proof}

Based on these results, we exhibit several examples of asymptotically level stable functions in the following. For $C\subseteq\mathbb{R}^n$, we denote by $\iota_C$ the indicator function on $C$. That is for any $x\in\mathbb{R}^n$,
$$
\iota_C(x):=\begin{cases}
0,&\mathrm{~if~} x\in C,\\
+\infty, &\mathrm{~else}.
\end{cases}
$$

\begin{example}\label{exa:alsExa}
Let $b\in\mathbb{R}^t$ and $A$ a ${t\times n}$ matrix. Examples of asymptotically level stable data fitting functions $\phi$ emerging in applications are provided in the following:
\begin{itemize}
\item [(i)] $\ell_q$ functions  with $q>0$:
\begin{equation}\label{eq:5}
\phi(x):=\|Ax-b\|_q^q.
\end{equation}
\item [(ii)]  Indicator functions on a compact set:
    $$
    \phi(x):=\iota_C(Ax)
    $$
where $C:=\{y\in\mathbb{R}^t: \|y-b\|_q\le \epsilon\}$, $\epsilon\ge 0$, $q>0$.
\item [(iii)] $\ell_0$ function:
\begin{equation}\label{eq:7}
\phi(x):=\|Ax-b\|_0.
\end{equation}
\item [(iv)] $\ell_2+\ell_1$ function (\cite{Dong-Ji-Li-Shen:acha:12}):
\begin{equation}\label{eq:6}
\phi(x, u):= \|Ax-b+u\|_2^2+\lambda'\|u\|_1,
\end{equation}
where $\lambda'>0$ is a parameter.
\item [(v)] $\ell_2+\ell_0$ function (
    \cite{Yan:L0Restoration}):
\begin{equation}
\phi(x, u):=\|Ax-b+u\|_2^2+\lambda'\|u\|_0.
\end{equation}
\end{itemize}
\end{example}
The power of asymptotically level stable functions is captured by the following theorem (Corollary 3.4.2 of \cite{Teboulle:asympt2003}), which plays a crucial role in the existence of optimal solutions to an optimization problem.
\begin{theorem}(Corollary 3.4.2 of \cite{Teboulle:asympt2003})\label{thm:alsExist}
Let $f:\mathbb{R}^n\rightarrow \bar{\mathbb{R}}$ be asymptotically level stable and bounded below. Then $f$ has at least one  global minimizer.
\end{theorem}

To end this section, we state an assumption on $\phi$ explicitly here for easy reference in the remainder of this paper.
\begin{hypothesis}\label{hypo:phi}
 The proper function $\phi:\mathbb{R}^n\rightarrow\bar{\mathbb{R}}$ is lower semicontinuous, asymptotically level stable and bounded below.
\end{hypothesis}
\section{Existence of optimal solutions to problems \eqref{model:l0} and \eqref{model:3}}
This section is devoted to  the existence of optimal solutions of problems \eqref{model:l0} and \eqref{model:3}. We obtain in this section that if  $\phi$ satisfies H\ref{hypo:phi}, then problems \eqref{model:l0} and  \eqref{model:3} have optimal solutions.

 By Theorem \ref{thm:alsExist}, if a function $f$ is asymptotically level stable and bounded below, then $f$ has a global minimizer. Thus, in order to show problems \eqref{model:l0} and  \eqref{model:3} have optimal solutions, it suffices to prove $\Phi$ and $\Psi_\gamma$ are asymptotically level stable and bounded below for any $\gamma>0$.

The next proposition reveals  when will $\Phi$ and $\Psi_\gamma$ are asymptotically level stable.

\begin{proposition}\label{prop:asl}
 Let $\phi:\mathbb{R}^n\rightarrow\bar{\mathbb{R}}$  satisfy H\ref{hypo:phi}. Let $\Phi$ and $\Psi_\gamma$ be defined by \eqref{model:l0} and \eqref{model:3} respectively. Then both $\Phi$ and $\Psi_\gamma$, for any $\gamma>0$, are asymptotically level stable and bounded below.
\end{proposition}
\begin{proof}
It is obvious that $\Phi$ and $\Psi_\gamma$ are   lower semicontinuous and proper since $\phi$ and $\psi_\gamma$ are lower semicontinuous, proper and $\mathrm{dom}(\psi_\gamma)=\mathbb{R}^m$. The lower boundedness of $\Phi$ and $\Psi_\gamma$  follows immediately from the lower boundedness of $\phi$, $\ell_0$ norm and $\psi_\gamma$.
According to Proposition \ref{prop:Asy+L0asl}, $\Phi$ is asymptotically level stable. We next try our best to prove $\Psi_\gamma$ is asymptotically level stable for any $\gamma>0$.

For any $\gamma>0$, $\lambda>0$ and $p>0$, let $\rho>0$ and $\{t^k: k\in\mathbb{N}\}$ be any bounded sequence of reals and  $\{x^k: k\in\mathbb{N}\}$ be any sequence   satisfying
\begin{equation}\label{eq:eq1}
x^k\in\mathrm{lev}(\Psi_\gamma, t^k), \|x^k\|_2\rightarrow +\infty, x^k\|x^k\|_2^{-1}\rightarrow \bar{x}\in\mathrm{ker}({\Psi_\gamma}_\infty).
\end{equation}
In order to get this proposition, we require to prove that there exists $k_0>0$ such that
\begin{equation}\label{eq:eq3}
\Psi_\gamma(x^k-\rho\bar{x})=\phi(x^k-\rho\bar{x})+\lambda\psi_\gamma(B(x^k-\rho\bar{x}))\le t^k
\end{equation}
holds for any $k\ge k_0$.

By Lemma \ref{lema:PhiInfty}, $\bar x\in\mathrm{ker}(\phi_\infty)$.  Set $\tau^k:=t^k-\lambda\psi_\gamma(Bx^k)$. Then $\{\tau^k: k\in\mathbb{N}\}$ is bounded because $\psi_\gamma$ is bounded.  Thus, \eqref{eq:eq1} implies
$$
x^k\in\mathrm{lev}(\phi, \tau^k), \|x^k\|_2\rightarrow +\infty, x^k\|x^k\|_2^{-1}\rightarrow \bar{x}\in\mathrm{ker}(\phi_\infty).
$$
Since $\phi$ is asymptotically level stable, we have that there exists $k_1>0$ such that
\begin{equation}\label{eq:0}
\phi(x^k-\rho\bar{x})\le \tau^k=t^k-\lambda\psi_\gamma(Bx^k)
\end{equation}
 for any $k\ge k_1$.

Finally, we proceed to proving there exists $k_0>0$ such that
\begin{equation}\label{eq:1}
\psi_\gamma(B(x^k-\rho\bar x))\le \psi_\gamma(Bx^k)
\end{equation}
holds for all $k\ge k_0$.
To this end, set $\Lambda:=\mathrm{supp}(B\bar x)$. Then $\varphi_\gamma((B(x^k-\rho\bar x))_i)=\varphi_\gamma((Bx^k)_i)$ for $i\in\Lambda^C$, where $\varphi_\gamma$ is defined by \eqref{def:varphi}.
We next focus on indexes in $\Lambda$. For $i\in\Lambda$, $\frac{(Bx^k)_i}{\|x^k\|_2}\rightarrow (B\bar x)_i$
due to $\frac{x^k}{\|x^k\|_2}\rightarrow \bar x$. Then $(Bx^k)_i\rightarrow\infty$ for all $i\in\Lambda$. Therefore, there exists $k_0>k_1$ such that $|(Bx^k)_i|\ge\frac{1}{\gamma^{1/p}}$ for all $i\in\Lambda$ and $k\ge k_0$. Immediately we have  $\varphi_\gamma((B(x^k-\rho\bar x))_i)\le\varphi_\gamma((Bx^k)_i)=1$ for all $i\in\Lambda$ and $k\ge k_0$.
Thus for any $k\ge k_0$,
\begin{eqnarray*}
\psi_\gamma(B(x^k-\rho\bar x))&=&\sum_{i=1}^m\varphi_\gamma((B(x^k-\rho\bar x))_i)\\
&=&\sum_{i\in\Lambda}\varphi_\gamma((B(x^k-\rho\bar x))_i)+\sum_{i\in\Lambda^c}\varphi_\gamma((B(x^k-\rho\bar x))_i)\\
&\le&\sum_{i\in\Lambda}\varphi_\gamma((Bx^k)_i)+\sum_{i\in\Lambda^c}\varphi_\gamma((Bx^k)_i)\\
&=&\psi_\gamma(Bx^k).
\end{eqnarray*}
Then, inequality \eqref{eq:0} and \eqref{eq:1}  imply \eqref{eq:eq3}. By definition \ref{def:als}, we get this proposition.
\end{proof}

Now, by applying  Proposition \ref{prop:asl} and Theorem \ref{thm:alsExist}, we establish the main result of this section in the next theorem.
\begin{theorem}\label{thm:SoulExist}
 Let $\phi$ satisfy H\ref{hypo:phi}. Then both the optimal solution sets to problems \eqref{model:l0} and \eqref{model:3} are not empty.
\end{theorem}

\section{Asymptotic approximation to problem \eqref{model:l0}}

In this section, we aim at showing that problem \eqref{model:3} provides asymptotic approximation for problem \eqref{model:l0} when $\phi$ is a level bounded function composed with a linear mapping.

We begin with an important lemma. We denote by ${\bf{0}}_{r\times t}$ (resp., $\mathbf{0}_n$) the $r\times t$ matrix (resp., $n$-dimensional vector) with all entries  $0$.
For an $r\times t$ matrix $A\neq \mathbf{0}_{r\times t}$, let  $\sigma_{min}(A)$ be the minimal nonzero singular value of $A$. For $\Lambda\subseteq \mathbb{N}_r$, let $A_\Lambda$ be the matrix formed by the rows of $A$ with indexes in $\Lambda$.  Similarly, for any $x\in\mathbb{R}^n$ and $\Lambda\subseteq \mathbb{N}_n$, we denote by $x_\Lambda$ the vector formed by the components of $x$ with indexes in $\Lambda$. For a set $S$, denote by $|S|$ the number of components of $S$ and let  $2^S$ collects all the nonempty subsets of $S$.  Then we define $\sigma:\mathbb{R}^{r\times t}\backslash\{\mathbf{0}_{r\times t}\}\rightarrow (0,+\infty)$ at $A\in\mathbb{R}^{r\times t}\backslash \{\mathbf{0}_{r\times t}\}$ as
\begin{equation}\label{def:sigma}
\sigma(A):=\min\{\sigma_{min}(A_\Lambda):{\Lambda\in 2^{\mathbb{N}_r}}, A_\Lambda\neq {\bf 0}_{|\Lambda|\times t}\}.
\end{equation}
\begin{lemma}\label{lema:3.2}
Let $\gamma_0>0$, $t\in\mathbb{N}_n$, $u\in\mathbb{R}^n$, $B\in\mathbb{R}^{m\times n}$ and $p>0$. Let $\psi_\gamma$ and $\sigma$ be defined by \eqref{def:psi} and \eqref{def:sigma} respectively. Let $\{\xi_i\in\mathbb{R}^n: i\in\mathbb{N}_t\}$ be an orthonormal  basis of a subspace $\mathbb{S}\subseteq \mathbb{R}^n$ and $\Xi$ be the matrix whose $i$-th column is $\xi_i$.  Suppose $B\Xi\neq \mathbf{0}_{m\times t}$. Then, for  any $w\in\mathbb{S}$, there exist  $w'\in\mathbb{S}$ satisfying
\begin{equation}\label{eq:10}
\|w'\|_2\le \frac{\sqrt{m}}{\sigma(B\Xi)}(\|Bu\|_\infty+\frac{1}{\gamma_0^{1/p}})
\end{equation}
such that
\begin{equation}\label{eq:11}
\psi_\gamma(B(u+w'))\le \psi_\gamma(B(u+w))
\end{equation}
 for any $\gamma\ge \gamma_0$.
\end{lemma}
\begin{proof}
Let  $w\in\mathbb{S}$ and set $\tau_0:=\frac{1}{\gamma_0^{1/p}}$. Then,
there exists $y\in\mathbb{R}^t$ such that $w=\Xi y$. Set $\Lambda:=\{i: |(B\Xi y)_i|\le \|Bu\|_\infty+\tau_0\}$. Thus, for any $i\in \Lambda^C$,  $|(Bu+B\Xi y)_i|\ge |(B\Xi y)_i|-|(Bu)_i|> \tau_0$. Then,  $\varphi_\gamma((Bu+Bw)_i)=1$  for any $\gamma\ge \gamma_0$ and any $i\in\Lambda^C$. Clearly, if $\Lambda=\emptyset$, $\psi_\gamma(B(u+w))=m$ is the maximal value of $\psi_\gamma$. In such a case, by
setting $w'=\mathbf{0}_n$, \eqref{eq:10} and \eqref{eq:11} hold. Next, we assume $\Lambda\neq \emptyset$.

Let $\eta_i\in\mathbb{R}^t$ be the $i$-th row of $(B\Xi)_{\Lambda}$ and $\mathbb{Y}:=\mathrm{span}\{\eta_i: i\in\mathbb{N}_{|\Lambda|}\}$. Let $y=y_\mathbb{Y}+y_{\mathbb{Y}^\perp}$ with $y_\mathbb{Y}\in\mathbb{Y}$ and $y_{\mathbb{Y}^\perp}\in\mathbb{Y}^\perp$. Then $\langle \eta_{i}, y_\mathbb{Y}\rangle=\langle \eta_{i},  y\rangle$. Thus, for $i\in\Lambda$ and $\gamma\ge \gamma_0$, there holds
\begin{equation}\label{eq:2.1}
\varphi_\gamma ((Bu+B\Xi y_\mathbb{Y})_i)=\varphi_\gamma ((Bu+B\Xi y)_i)=\varphi_\gamma ((Bu+Bw)_i).
 \end{equation}
For $i\in\Lambda^C$ and $\gamma\ge \gamma_0$,  we have
\begin{equation}\label{eq:2.2}
\varphi_\gamma ((Bu+B\Xi y_\mathbb{Y})_i)\le 1=\varphi_\gamma ((Bu+Bw)_i).
 \end{equation}
Set $w'=\Xi y_\mathbb{Y}$. Then \eqref{eq:2.1} and \eqref{eq:2.2} imply
\begin{eqnarray*}
\psi_\gamma(B(u+w'))&=&\psi_\gamma(Bu+B\Xi y_\mathbb{Y})\\
&=&\sum_{i\in\Lambda}\varphi_\gamma ((Bu+B\Xi y_\mathbb{Y})_i)+\sum_{i\in\Lambda^C}\varphi_\gamma ((Bu+B\Xi y_\mathbb{Y})_i)\\
&\le&\sum_{i\in\Lambda}\varphi_\gamma ((Bu+Bw)_i)+\sum_{i\in\Lambda^C}\varphi_\gamma ((Bu+Bw)_i)\\
&=&\psi_\gamma(B(u+w))
\end{eqnarray*}
holds for any $\gamma\ge \gamma_0$. Inequality \eqref{eq:11} follows immediately.

We next prove \eqref{eq:10}. We first show a trivial case when $(B\Xi)_{|\Lambda|}={\bf{0}}_{|\Lambda|\times t}$. In this case, $y_{\mathbb{Y}}={\bf{0}}_{t}$, therefore, $w'={\bf{0}}_n$. Then \eqref{eq:10} holds obviously. When $(B\Xi)_{|\Lambda|}\neq{\bf{0}}_{|\Lambda|\times t}$,  by the definition of $\Lambda$ and $\mathbb{Y}$,  $\|(B\Xi)_\Lambda y_\mathbb{Y}\|_2=\|(B\Xi)_\Lambda y\|_2\le \sqrt{m}(\|Bu\|_\infty+\tau_0)$.  Since $y_\mathbb{Y}\in \mathrm{range}((B\Xi)_\Lambda^\top)=\mathrm{ker}((B\Xi)_\Lambda)^\perp$, we have $$
\|w'\|_2=\|\Xi y_\mathbb{Y}\|_2=\|y_\mathbb{Y}\|_2\le \frac{\sqrt{m}}{\sigma_{min}((B\Xi)_\Lambda)}(\|Bu\|_\infty+\tau_0).$$
Finally, by the definition of $\sigma$, that is \eqref{def:sigma}, we get \eqref{eq:10}.
We complete the proof.
\end{proof}

We next establish a crucial lemma by taking advantage of Lemma \ref{lema:3.2}.
\begin{lemma}\label{lema:3.3}
Let $\gamma_0>0$,  $A\in\mathbb{R}^{r\times n}$ and $B\in\mathbb{R}^{m\times n}$.  Let $\Phi$ and $\Psi_\gamma$ be defined by \eqref{model:l0} and \eqref{model:3} respectively.  Let $\alpha\ge \inf\Phi$. Suppose the dimension of $\mathrm{ker}(A)$ is $t$. Let $\{\xi_i: i\in\mathbb{N}_t\}$ be an orthonormal basis of $\mathrm{ker}(A)$. Set $\Xi$ be the matrix whose $i$-th column is $\xi_i$. If $\phi:=f\circ A$ with $f:\mathbb{R}^r\rightarrow\bar{\mathbb{R}}$ proper, lower semicontinuous,  level bounded and bounded below,  then the following statements hold:
\begin{itemize}
\item [(i)] $S_1:=\{u\in\mathrm{ker}(A)^\perp: f(Au)\le \alpha\}$ is nonempty, closed and bounded. Consequently, $U:=\sup\{\|Bu\|_\infty: u\in S_1\}$ satisfies $0\le U< +\infty$.
\item [(ii)] Set $S_2:=\{w\in\mathrm{ker}(A): \|w\|_2\le \frac{\sqrt{m}}{\sigma(B\Xi)}(U+\frac{1}{\gamma_0^{1/p}})\}$ if $B\Xi\neq \mathbf{0}_{m\times t}$ and $S_2:=\{\mathbf{0}_n\}$ otherwise, where $\sigma$ is defined by \eqref{def:sigma}. Then for any $x\in\mathrm{lev}(\Psi_\gamma, \alpha)$, there exists $x'\in S_1+S_2$ such that $\Psi_\gamma(x')\le \Psi_\gamma(x)$ for any $\gamma\ge \gamma_0$.
\item [(iii)] $S_1+S_2$ is compact.
\end{itemize}
\end{lemma}
\begin{proof}
We first prove Item (i). From Proposition \ref{prop:leveBoun}, Theorem \ref{thm:SoulExist}, both problems \eqref{model:l0} and \eqref{model:3} have optimal solutions. Since $\alpha\ge \inf\Phi$ and $\|Bx\|_0\ge 0$ for any $x\in\mathbb{R}^n$, $\mathrm{lev}(f\circ A, \alpha)$ is nonempty. Thus $S_1$ is nonempty because any $x\in\mathrm{lev}(f\circ A, \alpha)$ can be decomposed as $x=u+w$ with $u\in\mathrm{ker}(A)^\perp$ and $w\in\mathrm{ker}(A)$. Since $f$ is level bounded, $\{Au: u\in S_1\}$ is bounded. Then $S_1$ is bounded due to the fact that $u\in\mathrm{ker}(A)^\perp$ for any $u\in S_1$. Therefore, $0\le U< +\infty$.

We next prove Item (ii). Set $\tau_0:=\frac{1}{\gamma_0^{1/p}}$. Since  $\Psi_\gamma\le \Phi$ for any $\gamma>0$,  $\mathrm{lev}(\Psi_\gamma, \alpha)$ is nonempty for any $\gamma>0$. Further, $\mathrm{lev}(\Psi_\gamma, \alpha)\subseteq \mathrm{lev}(f\circ A, \alpha)$ due to $\psi_\gamma\ge 0$.  Let $\gamma\ge\gamma_0$ and $x\in\mathrm{lev}(\Psi_\gamma, \alpha)$. Then $x$ can be uniquely decomposed as $x=u+w$, where $u\in\mathrm{ker}(A)^\perp$ and $w\in\mathrm{ker}(A)$. It is obvious that $u\in S_1$.
We first consider  the trivial case that $B\Xi=\mathbf{0}_{m\times t}$, which means that $\mathrm{ker}(A)\subseteq\mathrm{ker}(B)$. Then by setting  $x'=u$ one obtains $\Psi_\gamma(x')=\Psi_\gamma(x)$. Item (ii) follows immediately. We next discuss the case where $B\Xi\neq \mathbf{0}_{m\times t}$.
According to Lemma \ref{lema:3.2}, there exists $w'\in\mathrm{ker}(A)$ with $\|w'\|_2\le \frac{\sqrt{m}}{\sigma(B\Xi)}(\|Bu\|_\infty+\tau_0)$ such that $\psi_\gamma(B(u+w'))\le \psi_\gamma(B(u+w))=\psi_\gamma(Bx)$ for any $\gamma\ge \gamma_0$.
Set $x':=u+w'$. Obviously, $x'\in S_1+S_2$. Then for any $\gamma\ge \gamma_0$ we obtain
 \begin{eqnarray*}
\Psi_\gamma(x')&=&\phi(x')+\lambda\psi_\gamma(Bx')\\
&=&f(Au)+\lambda\psi_\gamma(B(u+w'))\\
&\le &f(A(u+w))+\lambda\psi_\gamma(B(u+w))\\
&=& \phi(x)+\lambda\psi_\gamma(Bx)\\
&=&\Psi_\gamma(x).
\end{eqnarray*}
Then we get Item (ii).

Item (iii) follows immediately from the fact that both $S_1$ and $S_2$ are nonempty, closed and bounded.
\end{proof}

With the help of Lemma \ref{lema:3.3}, we obtain the next proposition.
\begin{proposition}\label{prop:minCompact}
Let $\gamma_0>0$, $A\in\mathbb{R}^{r\times n}$, $B\in\mathbb{R}^{m\times n}$. Let $\Psi_\gamma$ be defined by \eqref{model:3}. If $\phi:=f\circ A$ with $f:\mathbb{R}^r\rightarrow \bar{\mathbb{R}}$ being proper, lower semicontinuous,  level bounded and bounded below, then there exists a compact set $S\subset\mathbb{R}^n$ such that for all $\gamma\ge \gamma_0$ there holds
$$
\min\{\Psi_\gamma(x): x\in\mathbb{R}^n\}=\min\{\Psi_\gamma(x): x\in S\}.
$$
\end{proposition}
\begin{proof}
From Proposition \ref{prop:leveBoun} and Theorem \ref{thm:SoulExist}, the optimal solution set of problem \eqref{model:3} is nonempty.
Let $\alpha\ge \inf \Phi$, where $\Phi$ is defined by \eqref{model:l0}.
 Then for any $\gamma>0$,
\begin{equation}\label{eq:12}
\min\{\Psi_\gamma(x): x\in\mathbb{R}^n\}=\min\{\Psi_\gamma(x): x\in\mathrm{lev}(\Psi_\gamma, \alpha)\}.
\end{equation}
By Item (ii) and Item (iii) of Lemma \ref{lema:3.3}, there exists a compact set $S\subset \mathbb{R}^n$ such that
\begin{equation}\label{eq:13}
\min\{\Psi_\gamma(x): x\in\mathrm{lev}(\Psi_\gamma, \alpha)\}=\min\{\Psi_\gamma(x): x\in S\}
\end{equation}
for any $\gamma\ge \gamma_0$.
Then this proposition follows from \eqref{eq:12} and \eqref{eq:13} immediately.
\end{proof}

Now, we are ready to present the main result of this section.
\begin{theorem}\label{thm:AsyCon}
Let $A\in\mathbb{R}^{r\times n}$ and $\{\gamma^k>0: k\in\mathbb{N}\}$ be an increasing sequence going to infinity. Let  $\Phi$  and $\Psi_\gamma$ be defined by \eqref{model:l0} and \eqref{model:3} respectively. If $\phi:=f\circ A$  with $f:\mathbb{R}^r\rightarrow \bar{\mathbb{R}}$ proper, lower semicontinuous,  level bounded and bounded below,  then the following statements hold:
\begin{itemize}
\item [(i)] $\min \Psi_{\gamma^k}\rightarrow \min\Phi$ as $k\rightarrow +\infty$.
\item [(ii)] $\lim\sup_k(\arg\min \Psi_{\gamma^k})\subseteq \arg\min \Phi$.
\item [(ii)] If $x^k$ is a global minimizer of $ \Psi_{\gamma^k}$ (an optimal solution  of problem \eqref{model:3} with $\gamma=\gamma^k$), then  any cluster point of $\{x^k: k\in\mathbb{N}\}$ is a global minimizer of $\Phi$ (an optimal solution of problem \eqref{model:l0}).
\end{itemize}
\end{theorem}
\begin{proof}
According to Proposition \ref{prop:leveBoun} and Theorem \ref{thm:SoulExist}, both problems \eqref{model:l0} and \eqref{model:3} have optimal solutions. Thus $\min\Psi_\gamma$ and $\min\Phi$ exist. We also obtain $\Psi_{\gamma^{k+1}}\ge\Psi_{\gamma^k}$ for any $k\in\mathbb{N}$ because $\{\gamma^k: k\in\mathbb{N}\}$ is an increasing sequence.
Therefore, $\{\Psi_{\gamma^k}: k\in\mathbb{N}\}$ epi-converges to $\Phi$ according to Proposition 7.4 of \cite{Teboulle:asympt2003} and $\Phi=\sup_{k\in\mathbb{N}}\Psi_{\gamma^k}$. Since $\{\gamma^k: k\in\mathbb{N}\}$ is an increasing sequence, $\gamma^k\ge {\gamma^1}>0$ for all $k\in\mathbb{N}$. By Proposition \ref{prop:minCompact}, there exists a bounded closed set $S\subset \mathbb{R}^n$ such that $\min\{\Psi_{\gamma^k}: x\in\mathbb{R}^n\}=\min\{\Psi_{\gamma^k}: x\in S\}$ for all $k\in\mathbb{N}$. Then through Theorem 7.31(a) of \cite{Rockafellar2004Variational}, Item (i) follows immediately.

By Theorem 7.31(b) of \cite{Rockafellar2004Variational}, Item (ii) follows since $-\infty<\min \Phi<+\infty$ and problem \eqref{model:3} has optimal solutions.

We next prove Item (iii). Let $\bar x$ be a cluster of $\{x^k: k\in\mathbb{N}\}$. Then
$$
\Phi(\bar x)\le\lim\inf _{k\rightarrow +\infty}\Psi_{\gamma^k}(x^k)=\min \Phi.
$$
The first inequality follows from epi-convergence and the last via Item (i). Thus, Item (iii) is obtained.
\end{proof}

Theorem \ref{thm:AsyCon} establishes that optimal solutions to problem \eqref{model:l0} can be asymptotically approximated by problem \eqref{model:3} provided that $\phi$ is a level bounded function composed by a linear mapping. We will further investigate in Section \ref{sec:exactEq} the exact equivalence between optimal solutions of problems \eqref{model:l0} and \eqref{model:3} when $\phi$ are some special data fitting functions. Before then we shall study in the next section the stability for problem \eqref{model:l0} with respect to the parameter $\lambda$. These results will play an important role in Section \eqref{sec:exactEq}.

\section{Stability for problem \eqref{model:l0}}\label{sec:Sta}
This section is devoted to the stability for problem \eqref{model:l0}, including behaviors of the optimal value and optimal solution set with respect to  changes of the parameter $\lambda$.  The stability for problem \eqref{model:l0} will serve as basis for discussions on exact approximations  in Section \ref{sec:exactEq}.  The idea of analysis is drawn from \cite{Zhang-Li:IP2016}, where the stability for a special non-composite  $\ell_0$ regularization problem is studied. However, the proofs in \cite{Zhang-Li:IP2016} can not be directly extended to investigating the stability for problem \eqref{model:l0}. Therefore, we decide to provide detailed proofs  in this section.

Suppose $\phi$ satisfies H\ref{hypo:phi} in the whole of this section.  We prove in this section that the optimal function value of problem \eqref{model:l0}  changes piecewise linearly as $\lambda$ varies. While the optimal solution set to problem \eqref{model:l0} is piecewise constant with respect to $\lambda$.

We begin with introducing the notion of marginal functions \cite{Teboulle:asympt2003}. Let $F:(0,+\infty)\rightarrow \mathbb{R}$ defined at $\lambda>0$ as
\begin{equation}\label{def:F}
F(\lambda):=\min\{\phi(x)+\lambda\|Bx\|_0: x\in\mathbb{R}^n\}.
\end{equation}
Obviously, for a fixed $\lambda>0$, $F(\lambda)$ is the optimal function value of problem \eqref{model:l0}. By Theorem \ref{thm:SoulExist}, $F$ is well defined as long as $\phi$ satisfies H\ref{hypo:phi}. The function $F$ is  called the marginal function of problem \eqref{model:l0}. We also define $\Omega:(0, +\infty)\rightarrow 2^{\mathbb{R}^n}$ at $\lambda>0$ as the optimal solution set of problem \eqref{model:l0}, that is
\begin{equation}\label{def:Omega}
\Omega(\lambda):=\arg\min\{\phi(x)+\lambda\|Bx\|_0: x\in\mathbb{R}^n\}.
\end{equation}
 It is clear that $\Omega$ is also well defined.

With the help of the notations $F$ and $\Omega$, our task in this section becomes exploring the properties of $F$ and $\Omega$. To this end, we require to establish several notations in the next subsection, which will play an important role in our analysis.

\subsection{Alternating minimization sequences}
In this subsection, we generate several important sequences
by alternatingly minimizing $\phi$ and $\|B\cdot\|_0$.
\begin{definition}\label{def:sRhoOmega}
Given $\phi$ satisfying H\ref{hypo:phi} and $B\in\mathbb{R}^{m\times n}$, the integer $L$, the sets $\{s_i\in \mathbb{N}: i \in\mathbb{N}_{L}^0\}$, $\{\rho_i\ge 0: i \in\mathbb{N}_{L}^0\}$, $\{\Omega_i\subseteq \mathbb{R}^n: i \in\mathbb{N}_{L}^0\}$ are defined by the following iteration
\begin{eqnarray*}
\mathrm{set~~}& i=0,\\ &\rho_0:=&\min\{\phi(x):x\in\mathbb{R}^n\},\\
&s_0:=&\min\{\|Bx\|_0: \phi(x)=\rho_0, x\in\mathbb{R}^n\},\\
&\Omega_0:=&\mathrm{arg}\min\{\|Bx\|_0: \phi(x)=\rho_0, x\in\mathbb{R}^n\},\\
\mathrm{while~~}&s_i>0 &\mathrm{~and~}  \mathrm{dom}(\phi)\cap \{x\in\mathbb{R}^n:\|Bx\|_0\le s_i-1\}\neq \emptyset\\
&\rho_{i+1}:=&\min\{\phi(x): \|Bx\|_0\le s_i-1, x\in\mathbb{R}^n\},\\
&s_{i+1}:=&\min\{\|Bx\|_0: \phi(x)=\rho_{i+1}, x\in\mathbb{R}^n\},\\
&\Omega_{i+1}:=&\arg\min\{\|Bx\|_0: \phi(x)=\rho_{i+1}, x\in\mathbb{R}^n\},\\
&i=&i+1,\\
\mathrm{end}&&\\
&L:=i.\\
\end{eqnarray*}
\end{definition}

We first show that Definition \ref{def:sRhoOmega} is well defined. It suffices to prove both the following two optimization problems
\begin{equation}\label{eq:subpConsL0}
\min\{\phi(x):\|Bx\|_0\le k, x\in\mathbb{R}^n\}
\end{equation}
and
\begin{equation}\label{eq:subpL0}
\min\{\|Bx\|_0: x\in S\}
\end{equation}
have optimal solutions when $\mathrm{dom}(\phi)\cap \{x\in\mathbb{R}^n: \|Bx\|_0\le k\}\neq \emptyset$ and $\emptyset \neq S\subseteq\mathbb{R}^n$.
 It is obvious that problem \eqref{eq:subpL0} has an optimal solution since the objective function is piecewise constant and has finite values. We will reveal that if $\phi$ satisfies H\ref{hypo:phi}, the optimal solution set of problem \eqref{eq:subpConsL0} is always nonempty as long as $\mathrm{dom}(\phi)\cap \{x\in\mathbb{R}^n: \|Bx\|_0\le k\}\neq \emptyset$. Before that, we require to recall the notion of asymptotically linear sets \cite{Teboulle:asympt2003}.
 \begin{definition}\label{def:asyLinSet}
 Let $C\subseteq \mathbb{R}^n$ be a nonempty closed set. Then $C$ is said to be an asymptotically linear set if for each $\rho>0$ and each sequence $\{x^k: k\in\mathbb{N}\}$ satisfying
 \begin{equation}\label{eq:5.1}
 x^k\in C, \|x^k\|_2\rightarrow +\infty, x^k\|x^k\|_2^{-1}\rightarrow \bar x,
 \end{equation}
 there exists $k_0\in\mathbb{N}$ such that
 $$
 x^k-\rho\bar x\in C
 $$
 for any $k\ge k_0$.
 \end{definition}
 \begin{remark}\label{remark:asyLiSet}
 According to Definitions \ref{def:asyLinSet} and \ref{def:als},
 $\iota_C$ is asymptotically level stable if $C\subseteq\mathbb{R}^n$ is asymptotically linear.
 \end{remark}

 The next lemma exhibits a class of asymptotically linear sets.
 \begin{lemma}\label{lema:asyLiSetScec5}
 Let $A$ be a $t\times n$ matrix and $b\in\mathbb{R}^t$. Then, for any $k\in\mathbb{N}_t^0$, $O:=\{x\in\mathbb{R}^n:\|Ax-b\|_0\le k\}$ is either empty or asymptotically linear.
 \end{lemma}
\begin{proof}
Suppose $O\neq \emptyset$. Let $\Lambda_i\subseteq \mathbb{N}_t$ satisfying $|\Lambda_i|=t-k$. One can check that there are $C_t^k$ such sets. For each $i\in\mathbb{N}_{C_t^k}$, set $O_i:=\{x\in\mathbb{R}^n: A_{\Lambda_i} x=b_{\Lambda_i}\}$. Then $O=\bigcup_{i=1}^{C_t^k}O_i$.
By Definition \ref{def:asyLinSet}, the union of a finite number of asymptotically linear sets is also an asymptotically linear set. Therefore, in order to get this lemma,  it suffices to prove $O_i$ is either empty or  asymptotically linear for any $i\in\mathbb{N}_{C_t^k}$.

Let $i\in\mathbb{N}_{C_t^k}$ satisfying $O_i\neq\emptyset$ and set $C:=O_i$. Let $\rho>0$ and $\{x^k: k\in\mathbb{N}\}$ be any  sequence satisfying \eqref{eq:5.1}. Then we have $A_{\Lambda_i}\bar x=\mathbf{0}_{t-k}$. Thus $A_{\Lambda_i}(x^k-\rho\bar x)=A_{\Lambda_i}x^k=b_{\Lambda_i}$, that is $x^k-\rho\bar x\in C$ for any $k\in\mathbb{N}$. Therefore, $O_i$ is asymptotically linear. We then obtain this lemma immediately.
\end{proof}
\begin{lemma}\label{lema:L0constSolu}
Given $\phi$ satisfying H\ref{hypo:phi},  $B\in\mathbb{R}^{m\times n}$ and $k\in\mathbb{N}_m^0$, the optimal solution set to problem \eqref{eq:subpConsL0} is nonempty if $\mathrm{dom}(\phi)\cap \{x\in\mathbb{R}^n:\|Bx\|_0\le k\}\neq \emptyset$.
\end{lemma}
\begin{proof}
According to the analytic representation \eqref{eq:AsyFunAna} of the asymptotic function, we obtain that  $\phi_\infty(d)\ge 0$ for all $d\in\mathbb{R}^n$ since $\phi$ is bounded below. By Lemma \ref{lema:asyLiSetScec5}, $C:=\{x\in\mathbb{R}^n: \|Bx\|_0\le k\}$ is asymptotically linear. Then, from Proposition 3.3.3 in \cite{Teboulle:asympt2003}, $\phi+\iota_C$ is asymptotically level stable due to the fact that $\phi$ is asymptotically level stable. Then, problem \eqref{eq:subpConsL0} has an optimal solution by Theorem \ref{thm:alsExist}.
\end{proof}

 From Lemma \ref{lema:L0constSolu}, Definition \ref{def:sRhoOmega} is well defined. We present this result in the following proposition.

\begin{proposition}\label{prop:rhoSOmegaExist}
 Given $\phi$ satisfying H\ref{hypo:phi} and $B\in\mathbb{R}^{m\times n}$, the integer $L$ and $s_i, \rho_i, \Omega_i$ for $i\in\mathbb{N}_{L}^0$ are well defined by Definition \ref{def:sRhoOmega}.
\end{proposition}

We then provide some properties of $L$ and $s_i, \rho_i, \Omega_i$ for $ i \in\mathbb{N}_{L}^0$ defined by Definition \ref{def:sRhoOmega} in the following proposition.
\begin{proposition}\label{prop:rhoSOmega}
Given $\phi$ satisfying H\ref{hypo:phi} and $B\in\mathbb{R}^{m\times n}$, let $L$ and $s_i, \rho_i, \Omega_i$ for $i\in\mathbb{N}_{L}^0$ be defined by Definition \ref{def:sRhoOmega}. Then the following statements hold:
\begin{itemize}
\item [(i)] $0\le L\le m$.
\item [(ii)] $m\ge s_0>s_1>\dots>s_L\ge 0$.
\item [(iii)] $\rho_0<\rho_1<\dots<\rho_L$.
\item [(iv)]  $\Omega_i\neq \emptyset$ and $\Omega_i\cap \Omega_j=\emptyset$, for $i\neq j$, $i, j\in\mathbb{N}_L^0$.
\item [(v)] $\Omega_i=\{x: \|Bx\|_0=s_i, \phi(x)=\rho_i, x\in\mathbb{R}^n\}$ for $i \in\mathbb{N}_{L}^0$.
\item [(vi)] $\mathrm{dom}(\phi)\cap\{x\in\mathbb{R}^n: \|Bx\|_0\le s_L-1\}=\emptyset$.
\end{itemize}
\end{proposition}
We omit the proof of Proposition \ref{prop:rhoSOmega} here since all the results follow immediately from  Definition \ref{def:sRhoOmega}.

With the help of Definition \ref{def:sRhoOmega},   the Euclid space $\mathbb{R}^n$ can be partitioned into $L+2$ sets: $\{x\in\mathbb{R}^n: \|Bx\|_0\ge s_0\}$, $\{x\in\mathbb{R}^n: s_{i+1}\le\|Bx\|_0\le s_i-1\}$, $i  \in\mathbb{N}_{L-1}^0$ and $\{x\in\mathbb{R}^n: \|Bx\|_0\le s_L-1\}$ (if $s_L=0$ this set is $\emptyset$).  Note that $\mathrm{dom}(\phi)\cap\{x\in\mathbb{R}^n:\|Bx\|_0\le s_L-1\}=\emptyset$. Therefore, in order to establish the optimal solutions to problem \eqref{model:l0}, we first explore the  optimal solutions to problems
\begin{equation}\label{eq:subPartition1}
\min\{\phi(x)+\lambda\|Bx\|_0: \|Bx\|_0\ge s_0, x\in\mathbb{R}^n\},
\end{equation}
and
\begin{equation}\label{eq:subPartition2}
\min\{\phi(x)+\lambda\|Bx\|_0: s_{i+1}\le\|Bx\|_0\le s_i-1, x\in\mathbb{R}^n\}
\end{equation}
for all $i \in\mathbb{N}_{L-1}^0$. We present the desired results in the following lemma.
\begin{lemma}\label{lema:partitionSolution}
Given $\phi$ satisfying H\ref{hypo:phi} and $B\in\mathbb{R}^{m\times n}$, let $L$ and $s_i, \rho_i, \Omega_i$ for $i\in\mathbb{N}_{L}^0$ be defined by Definition \ref{def:sRhoOmega}. Then  for any $\lambda>0$, the following statements hold:
\begin{itemize}
\item [(i)] The optimal solution sets to problems \eqref{eq:subPartition1} and \eqref{eq:subPartition2} with $i  \in\mathbb{N}_{L-1}^0$ are not empty.
\item [(ii)] The optimal value of problem \eqref{eq:subPartition1} is $\rho_0+\lambda s_0$ and the optimal solution set to problem \eqref{eq:subPartition1} is $\Omega_0$.
\item [(iii)] The optimal value of problem \eqref{eq:subPartition2} is $\rho_{i+1}+\lambda s_{i+1}$ and the optimal solution set to problem \eqref{eq:subPartition2} is $\Omega_{i+1}$, for $i  \in\mathbb{N}_{L-1}^0$.
\end{itemize}
\end{lemma}
\begin{proof}
We only need to prove Items (ii) and (iii) since they imply Item (i).

We first prove Item (ii). It is obvious that restricted to the set $\{x\in\mathbb{R}^n: \|Bx\|_0\ge s_0\}$, the minimal value of the last term $\|Bx\|_0$ is $s_0$ and can be attained at any $x\in\Omega_0$. By Definition \ref{def:sRhoOmega}, the minimal value of $\phi$ is $\rho_0$ and can be attained at any $x\in\Omega_0$. Therefore, the optimal value of problem \eqref{eq:subPartition1} is $\rho_0+\lambda s_0$. Let $\Omega^*$ be the optimal solution set of problem \eqref{eq:subPartition1}. Clearly, $\Omega_0\subseteq \Omega^*$. We then try to prove $\Omega^*\subseteq \Omega_0$. It suffices to prove $\|Bx\|_0=s_0$ for any $x\in\Omega^*$. If not, there exists  $x^*\in\Omega^*$ such that $\|Bx^*\|_0\neq s_0$. Clearly,  $\|Bx^*\|_0\ge s_0+1$. Then the objective function value at $x^*$ is no less than $\rho_0+\lambda(s_0+1)$ due to the definition of $\rho_0$, contradicting the fact that $x^*$ is an optimal solution of problem \eqref{eq:subPartition1}. Then we get Item (ii).
Item (iii) can be obtained similarly, we omit the details here.
\end{proof}

For convenient presentation, we define $f_i:(0, +\infty)\rightarrow \mathbb{R}$, for $i \in\mathbb{N}_{L}^0$, at $\lambda> 0$ as
\begin{equation}\label{def:fi}
f_i(\lambda):=\rho_i+\lambda s_i.
\end{equation}
The next theorem expresses the optimal function value and the optimal solution set to problem \eqref{model:l0} by $f_i$ and $\Omega_i$, $i\in\mathbb{N}_L^0$, respectively.
\begin{theorem}\label{thm:exist}
 Given $\phi$ satisfying H\ref{hypo:phi} and $B\in\mathbb{R}^{m\times n}$, let $L$ and $s_i, \rho_i, \Omega_i$ for $i\in\mathbb{N}_{L}^0$ be defined by Definition \ref{def:sRhoOmega}.
 Let  $f_i:(0, +\infty)\rightarrow\mathbb{R}$ for $i\in\mathbb{N}_L^0$ be defined by \eqref{def:fi}. Then, the following statements hold for any fixed $\lambda>0$:
 \begin{itemize}
 \item [(i)] The optimal value of problem \eqref{model:l0} is $\min\{f_i(\lambda): i \in\mathbb{N}_{L}^0\}$.
\item [(ii)] $\bigcup_{i\in\Lambda^*}\Omega_{i}$ is  the optimal solution set of problem \eqref{model:l0}, where $\Lambda^*:=\arg\min\{f_i(\lambda): i \in\mathbb{N}_{L}^0\}$.
    \end{itemize}
\end{theorem}
We omit the proof here since
it is a direct result of the fact that  $\mathbb{R}^n=\bigcup_{i=0}^{L-1}\{x\in\mathbb{R}^n:s_{i+1}\le\|x\|_0\le s_{i}-1\}\bigcup\{x\in\mathbb{R}^n:\|x\|_0\ge s_0\}\bigcup\{x\in\mathbb{R}^n:\|x\|_0\le s_L-1\}$, Item (vi) of proposition \ref{prop:rhoSOmega} and Lemma \ref{lema:partitionSolution}.
\subsection{Properties of $F$ and $\Omega$}\label{sec:subSta}
Based on the previous subsection, this subsection focuses on  the stability to parameter $\lambda$ for problem \eqref{model:l0}. We shall study  properties of the marginal function $F$ and the optimal solution set $\Omega$, defined by \eqref{def:F} and \eqref{def:Omega} respectively, of problem \eqref{model:l0}. We prove in this subsection that $F$ is piecewise linear, while $\Omega$ is piecewise constant. The proof in this subsection is very similar to that used in Section 3.2 in \cite{Zhang-Li:IP2016}. Therefore, we only present the main results of this subsection and one can refer to Section 3.2 and Appendix of \cite{Zhang-Li:IP2016} for detailed proofs.

According to Theorem \ref{thm:exist},
$F(\lambda)=\min\{f_i(\lambda), i \in\mathbb{N}_{L}^0\}$ with $f_i$ defined by \eqref{def:fi}. Obviously,  each $f_i$ for $i \in\mathbb{N}_{L}^0$ is a line with slop $s_i$ and intercept $\rho_i$. Therefore, by Items (ii) and (iii) of Proposition \ref{prop:rhoSOmega}, it is easy to deduce that $F$ is continuous and piecewise linear. We will utilize the following  iteration procedure to find the minimal value of $f_i$ for $i \in\mathbb{N}_{L}^0$, that is, the marginal function $F$.
\begin{definition}\label{def:lambda_i}
Given $\phi$ satisfying H\ref{hypo:phi} and $B\in\mathbb{R}^{m\times n}$, let $L$ and $s_i, \rho_i, \Omega_i$ for $i\in\mathbb{N}_{L}^0$ be defined by Definition \ref{def:sRhoOmega}. Then the integer $K$, $\{t_i\in \mathbb{N}_{L}^0:i\in\mathbb{N}_{K}^0\}, \{\lambda_i>0: i \in\mathbb{N}_{K}^0\}, \{\Lambda_i\subseteq \mathbb{N}_{L}^0: i=-1,0,\dots, K-1\}$ are defined by the following iteration
\begin{eqnarray*}
\mathrm{set~~}& i=0, &\Lambda_{-1}:=\{L\}, \\
\mathrm{while~~}&0\not\in\Lambda_{i-1} &\\
&t_{i}:=&\min \Lambda_{i-1},\\
&\lambda_i:=&\max\{\frac{\rho_{t_i}-\rho_j}{s_j-s_{t_i}}: j=0,1,\dots, t_i-1\},\\
&\Lambda_i:=&\arg\max\{\frac{\rho_{t_i}-\rho_j}{s_j-s_{t_i}}: j=0,1,\dots, t_i-1\},\\
&i=&i+1,\\
\mathrm{end~~}&&\\
&K:=i,& t_K:=0, \lambda_K:=0.\\
\end{eqnarray*}
\end{definition}

To understand the above definition, one can refer to Example 3.6 of \cite{Zhang-Li:IP2016}.
The following proposition provides some basic properties of $K$, $\{\Lambda_i\subseteq \mathbb{N}_{L}^0: i=-1,0,\dots, K-1\}$ and $t_i, \lambda_i$ for $i\in\mathbb{N}_{K}^0$ by Definition \ref{def:lambda_i}.
\begin{proposition}\label{prop:lambda_i}
Given $\phi$ satisfying H\ref{hypo:phi} and $B\in\mathbb{R}^{m\times n}$, let $L$ and $s_i, \rho_i, \Omega_i$ for $i\in\mathbb{N}_{L}^0$ be defined by Definition \ref{def:sRhoOmega}. Let  $K$, $\{\Lambda_i\subseteq \mathbb{N}_{L}^0: i=-1,0,\dots, K-1\}$ and $t_i, \lambda_i$ for $i\in\mathbb{N}_{K}^0$ be  defined by  Definition \ref{def:lambda_i}. Then the following statements hold:
\begin{itemize}
\item [(i)] $0\le K\le L$, in particular, if $L\ge 1$ then $K\ge 1$.
\item [(ii)] $L=t_0>t_1>\dots>t_K=0$.
\item [(iii)] $\lambda_0>\lambda_1>\dots>\lambda_K=0$.
\item [(iv)] $\Lambda_i\neq \emptyset$ and $\Lambda_i\cap\Lambda_j=\emptyset$, for all $i\neq j$, $i, j=-1, 0,\dots,  K-1$.
\end{itemize}
\end{proposition}

The main results of this subsection are presented in the following theorem.
\begin{theorem}\label{thm:stability}
Given $\phi$ satisfying H\ref{hypo:phi} and $B\in\mathbb{R}^{m\times n}$, let $L$ and $s_i, \rho_i, \Omega_i$ for $i\in\mathbb{N}_{L}^0$ be defined by Definition \ref{def:sRhoOmega}. Let  $K$, $\{\Lambda_i\subseteq \mathbb{N}_{L}^0: i=-1,0,\dots, K-1\}$ and $t_i, \lambda_i$ for $i\in\mathbb{N}_{K}^0$ be  defined by  Definition \ref{def:lambda_i}. Let $f_i$ for $i \in\mathbb{N}_{L}^0$, $F$ and $\Omega$ be defined by \eqref{def:fi}, \eqref{def:F} and \eqref{def:Omega} respectively.
\begin{itemize}
\item [(i)] If $L=0$, then $F(\lambda)=\rho_0+\lambda s_0$ and $\Omega(\lambda)=\Omega_0$ for any $\lambda>0$.
\item [(ii)] If $L\ge 1$, then
$$F(\lambda)=\begin{cases}
f_0(\lambda), &\mathrm{~~if~~}\lambda\in(0, \lambda_{K-1}],\\
f_{t_i}(\lambda),& \mathrm{~~if~~}\lambda\in (\lambda_i, \lambda_{i-1}], i\in\mathbb{N}_{K-1},\\
f_L(\lambda),& \mathrm{~~if~~}\lambda\in(\lambda_0, +\infty).
\end{cases}$$
and
 $$
\Omega(\lambda)=\begin{cases}
\Omega_0,& \mathrm{~~if~~}\lambda\in(0, \lambda_{K-1}),\\
\Omega_{t_i}, & \mathrm{~~if~~}\lambda\in (\lambda_i, \lambda_{i-1}), i\in\mathbb{N}_{K-1},\\
\Omega_L,& \mathrm{~~if~~}\lambda\in(\lambda_0, +\infty),\\
\bigcup_{k\in\Lambda_i}\Omega_k\cup\Omega_{t_i}, & \mathrm{~~if~~} \lambda=\lambda_i, i\in\mathbb{N}_{K-1}^0.
\end{cases}
$$
\item [(iii)] $F$ is continuous, piecewise linear, nondecreasing and concave.
\end{itemize}
\end{theorem}

A direct consequence of Theorem \ref{thm:stability} is stated below.
\begin{corollary}\label{crol:stability}
Under the assumptions of Theorem \ref{thm:stability},  the following statements hold:
\begin{itemize}
\item [(i)] If $\lambda', \lambda''\in (\lambda_0, +\infty)$ or $\lambda', \lambda''\in (\lambda_i, \lambda_{i-1})$, $i\in\mathbb{N}_K$, then $\|Bx\|_0=\|By\|_0$ and $\phi(x)=\phi(y)$ hold for any $x\in\Omega(\lambda')$ and any $y\in\Omega(\lambda'')$.
\item [(ii)] If $\lambda'<\lambda''$, then $\|Bx\|_0\ge \|By\|_0$ and $\phi(x)\le \phi(y)$ hold for any $x\in\Omega(\lambda')$ and any $y\in\Omega(\lambda'')$.
\end{itemize}
\end{corollary}

From Theorem \ref{thm:stability}, the optimal value of problem \eqref{model:l0} changes piecewise linearly while the optimal solution set to problem \eqref{model:l0} changes piecewise constantly as the parameter $\lambda$ varies. In addition, by Corollary \ref{crol:stability}, the optimal values of both the first and second terms of \eqref{model:l0}  are piecewise constant with respect to changes in the parameter $\lambda$.
\section{Exact approximation to problem \eqref{model:l0}}\label{sec:exactEq}
In this section, we explore the exact approximation to problem  \eqref{model:l0} by problem \eqref{model:3}.
We establish  two cases where  problems \eqref{model:l0} and \eqref{model:3} share the same optimal solution set provided that $\gamma>\gamma^*$ for some $\gamma^*>0$.


\subsection{When $\phi$ is the indicator function on an asymptotically linear  set}\label{subsec:IndicatorEquiv}
In this subsection, we consider the case where $\phi$ is the indicator function on an asymptotically linear  set (see Definition \ref{def:asyLinSet}) and derive exact approximation results regarding optimal solution sets of problems \eqref{model:l0} and \eqref{model:3}.

Let $C\subseteq\mathbb{R}^n$ and $\phi$ be the indicator function on $C$.  Problems \eqref{model:l0} and \eqref{model:3} become
\begin{equation}\label{eq:6.1.0}
\min\{\|Bx\|_0: x\in C\},
\end{equation}
and
\begin{equation}\label{eq:6.1.1}
\min\{\psi_\gamma(Bx): x\in C\}
\end{equation}
respectively.
By \eqref{eq:6.3.1}, for any $p>0$, problem \eqref{eq:6.1.1} can be equivalently reformulated as
\begin{equation}\label{eq:6.1.2}
\min\{\|v\|_0+{\gamma}\|Bx-v\|_p^p: x\in C, v\in\mathbb{R}^m\}.
\end{equation}
The following theorem concerns optimal solutions to problems \eqref{eq:6.1.0} and \eqref{eq:6.1.1}.

\begin{theorem}\label{thm:ExaEqvInd}
Let $C\subseteq\mathbb{R}^n$ be asymptotically linear. Then there exists a $\gamma^*>0$ such that problems \eqref{eq:6.1.0} and \eqref{eq:6.1.1} share the same optimal solution set whenever $\gamma>\gamma^*$.
\end{theorem}
\begin{proof}
Since $C$ is asymptotically linear, $\phi:=\iota_C$ is asymptotically level stable by Remark \ref{remark:asyLiSet}. Then problems \eqref{eq:6.1.0} and \eqref{eq:6.1.1} have optimal solutions from Theorem \ref{thm:SoulExist}.
Let $\Omega$ be the optimal solution set to problem \eqref{eq:6.1.0}, that is, $\Omega:=\arg\min\{\|Bx\|_0: x\in C\}$. It is clear that  $\tau\iota_C=\iota_C$ for any $\tau>0$. With the indicator function $\iota_C$, problem \eqref{eq:6.1.1} can be also written as
\begin{equation}\label{eq:6.1.4}
\min\{\iota_C(x)+\|Bx-v\|_p^p+\frac{1}{\gamma}\|v\|_0: x\in\mathbb{R}^n, v\in\mathbb{R}^m\}.
\end{equation}
Therefore, it suffices to prove that there exists a $\gamma^*>0$ such that
\begin{equation}\label{eq:6.1.5}
\Omega=\{x: (x, v)\in\tilde\Omega\} \mathrm{~for~} \gamma>\gamma^*,
 \end{equation}
 where $\tilde\Omega$ is the optimal solution set to \eqref{eq:6.1.4}.

In fact, problem \eqref{eq:6.1.4} can be cast into problem \eqref{model:l0} by setting
$$\tilde x:=(x, v), ~~\tilde \phi(\tilde x):=\iota_C(x)+\|Bx-v\|_p^p,~~\tilde B:=[\mathbf{0}_{m\times n}~~ I_{m\times m}]~~\mathrm{ and}~~\tilde \lambda:=\frac{1}{\gamma}.$$
Here, $\tilde \phi$ is asymptotically level stable according to Propositions \ref{prop:leveBoun}, Item (c) of Proposition 3.3.3 in \cite{Teboulle:asympt2003} and the asymptotic linearity of $C$. Therefore, Definitions \ref{def:sRhoOmega}, \ref{def:lambda_i} and  Theorem   \ref{thm:stability} can be applied to the problem
\begin{equation}\label{eq:16}
\min\{\tilde \phi(\tilde x)+\tilde\lambda\|\tilde B\tilde x\|_0: \tilde x\in\mathbb{R}^n\times\mathbb{R}^m\}.
 \end{equation}Given $\tilde \phi$ and $\tilde B\in\mathbb{R}^{m\times{(m+n)}}$, let $\rho_0, \Omega_0$  be defined by Definition \ref{def:sRhoOmega}.
  By Theorem   \ref{thm:stability}, there exits $\lambda^*>0$ such that  the optimal solution set of problem \eqref{eq:16} is $\Omega_0$  as $0<\tilde\lambda<\lambda^*$. From Definition \ref{def:sRhoOmega},
$$\rho_0=\min\{ \tilde \phi(\tilde x): \tilde x\in\mathbb{R}^n\times \mathbb{R}^{m}\}=\min\{\|Bx-v\|_p^p: x\in C, v\in\mathbb{R}^m\}.$$
Thus, we have $\rho_0=0$ and $\{\tilde x: \tilde \phi(\tilde x)=\rho_0\}=\{(x, v): Bx=v, x\in C\}$. Then $$\Omega_0=\arg\min\{\|v\|_0: Bx=v, x\in C\}=\arg\min\{\|Bx\|_0: Bx=v, x\in C\}.$$
 Therefore, we have \begin{equation}\label{eq:6.1.7}
 \tilde \Omega=\Omega_0 \mathrm{~if~} 0<\lambda=1/\gamma<\lambda^*.
  \end{equation}  It is Obvious that \begin{equation}\label{eq:6.1.6}
\Omega=\{x: (x, v)\in\Omega_0\}.
 \end{equation}
  By setting $\gamma^*=1/\lambda^*$, \eqref{eq:6.1.5} follows from \eqref{eq:6.1.7} and \eqref{eq:6.1.6}.
  We immediately obtain this theorem.
 \end{proof}

In the following lemma, we present a class of asymptotically linear sets encountered in applications frequently.
\begin{lemma}\label{lema:asyLinSetInpainting}
Let $A\in\mathbb{R}^{t\times n}$ and $C'\subset\mathbb{R}^t$  be a closed and bounded nonempty set. Then $C:=\{x\in\mathbb{R}^n: Ax\in C'\}$ is either empty or asymptotically linear.
\end{lemma}
\begin{proof}
Suppose $C\neq\emptyset$. It is obvious that $C$ is nonempty and closed. Let $\rho>0$ and $\{x^k: k\in\mathbb{N}\}$ satisfy \eqref{eq:5.1}. By Definition \ref{def:asyLinSet}, we require to prove there exits $k_0\in\mathbb{N}$ such that for all $k\ge k_0$, $x^k-\rho \bar x\in C$, that is, $A(x^k-\rho\bar x)\in C'$. It suffices to prove $\bar x\in \mathrm{ker}(A)$.

Since $Ax^k\in C'$, $\frac{x^k}{\|x^k\|_2}\rightarrow\bar x$, $\|x^k\|_2\rightarrow+\infty$ and $C'$ is bounded, one can deduce that $A\bar x=\mathbf{0}_t$. Immediately, we get this lemma.
\end{proof}

Below we give an concrete example of problems \eqref{eq:6.1.0} and \eqref{eq:6.1.2} where the set $C$ is asymptotically linear.

\begin{example}
Let $A\in\mathbb{R}^{t\times n}$, $b\in\mathbb{R}^t$, $B\in\mathbb{R}^{m\times n}$, $p, q>0$ and $\epsilon\ge 0$. Let $\Omega$ and $\tilde \Omega$ be defined by
$$
\Omega:=\arg\min\{\|Bx\|_0: \|Ax-b\|_q\le \epsilon, x\in\mathbb{R}^n\}
$$
and
$$
\tilde\Omega:=\arg\min\{\|v\|_0+\gamma\|Bx-v\|_p^p: \|Ax-b\|_q\le \epsilon, x\in\mathbb{R}^n, v\in\mathbb{R}^m \}.
$$
By Lemma \ref{lema:asyLinSetInpainting} and Theorem \ref{thm:ExaEqvInd}, there exists a $\gamma^*>0$ such that $\Omega=\{x: (x, v)\in\tilde\Omega\}$ whenever $\gamma>\gamma^*$.
\end{example}

We can further give the exact expression of $\gamma^*$ in the previous theorem when $C$ is  compact. We present this result in the next theorem.

\begin{theorem}\label{thm:ExaEqvInd2}
Let $p>0$, $\mathbf{0}_{m\times n}\neq B\in\mathbb{R}^{m\times n}$. Let $C\subset\mathbb{R}^n$ be a closed and bounded nonempty set. Let $\Omega^*$ and $k^*\in\mathbb{N}$ be the optimal solution set and the optimal function value of problem \eqref{eq:6.1.0} respectively.    Set $\tau:=\inf\{\mathrm{the~~} k^*\mathrm{-th~~ largest ~~value~~of~~}\{|(Bx)_i|: i\in\mathbb{N}_m\}: x\in C\}$. Then the following statements hold:
\begin{itemize}
\item [(i)] $\tau>0$.
\item [(ii)]Problems \eqref{eq:6.1.0} and \eqref{eq:6.1.1} share the same optimal solution set whenever $\gamma>\gamma^*$, where $\gamma^*:=1/\tau^p$
    \end{itemize}
\end{theorem}
\begin{proof}
We first prove Item (i).  Let $g, g_i: C \rightarrow\mathbb{R}$ for $i\in\mathbb{N}_m$ defined at $x\in C$ as  $g_i(x):=|(Bx)_i|$  and $g(x):=\mathrm{the~~} k^*\mathrm{-th~~largest~~value~~of ~~}\{g_i(x): i\in\mathbb{N}_m\}$. Obviously, $\tau=\inf_C g$. Since $g_i$ is continuous for any $i\in\mathbb{N}_m$, $g$ is proper and continuous. Moreover, $g(x)>0$ for any $x\in C$ because $k^*$ is the optimal function value of problem \eqref{eq:6.1.0}. Then problem
$$
\min\{g(x): x\in C\}
$$
has at least one optimal solution since $g$ is continuous and $C$ is closed and bounded. Therefore, $\tau=\inf_C g$ implies $\tau>0$.

In order to prove Item (ii), we first show $\psi_\gamma(Bx)\ge k^*$ for any $x\in C$ whenever $\gamma>\gamma^*$.
For any $x\in C$, set $\Lambda':=\{j\in\mathbb{N}_m: g_j(x)\ge g(x)\}$.  Clearly, $|\Lambda'|\ge k^*$ by the definition of  $g$. Then, for $i\in\Lambda'$,
$\gamma|(Bx)_i|^p=\gamma g_i^p(x)\ge g^p(x)/\tau^p\ge 1$. Thus, $\varphi_\gamma((Bx)_i)=1$ for $i\in\Lambda', \gamma>\gamma^*$.  Therefore,
 $\psi_\gamma(Bx)\ge \sum_{i\in\Lambda'}\varphi((Bx)_i)=|\Lambda'| \ge k^*$ for any $x\in C$.

We then prove
 any $x^*\in\Omega^*$ is an optimal solution to problem \eqref{eq:6.1.1} whenever $\gamma>\gamma^*$. It suffices to prove $\psi_\gamma(Bx^*)=k^*$. Let $\Lambda:=\mathrm{supp}(Bx^*)$. Thus $|\Lambda|=k^*$. Then, for any $\gamma>\gamma^*$,  $\gamma|(Bx^*)_i|^p>g_i(x^*)^p/\tau^p\ge g(x^*)/\tau^p \ge 1$ for $i\in\Lambda$.  Therefore, as $\gamma>\gamma^*$, $\varphi_\gamma((Bx^*)_i)=1$ for $i\in\Lambda$ and $\varphi_\gamma((Bx^*)_i)=0$ for $i\in\Lambda^C$. Thus $\psi_\gamma(Bx^*)=k^*$ whenever $\gamma>\gamma^*$. Then, the optimal function value of problem \eqref{eq:6.1.1} is $k^*$ and $x^*$ is an optimal solution to problem \eqref{eq:6.1.1}.

We finally show any optimal solution $\hat x$ of problem \eqref{eq:6.1.1} with $\gamma>\gamma^*$ is an optimal solution of problem \eqref{eq:6.1.0}. It suffices to prove $\|B\hat x\|_0=k^*$.
Since the optimal function value of \eqref{eq:6.1.1} is $k^*$, $\psi_\gamma(B\hat x)=k^*$.
We then prove $\|B\hat x\|_0=k^*$ by contradiction. Since $k^*$ is the optimal function value of problem \eqref{eq:6.1.0}, $\|B\hat x\|_0\ge k^*$. Thus $\|B\hat x\|_0\neq k^*$ implies $\|B\hat x\|_0>k^*$.  Suppose $\|B\hat x\|_0>k^*$. Set $\Lambda':=\{j\in\mathbb{N}_m: g_j(\hat x)\ge g(\hat x)\}$
and $\Lambda'':=\mathrm{supp}(B\hat x)$. It is obvious that $|\Lambda'|\ge k^*$, $|\Lambda''|=\|B\hat x\|_0>k^*$.  If $|\Lambda'|>k^*$, then by the previous analysis $\psi_\gamma(B\hat x)\ge |\Lambda'|>k^*$, which contradicts that the optimal function value of problem \eqref{eq:6.1.1} is $k^*$. Else if $|\Lambda'|=k^*$, then $\psi_\gamma(B\hat x)=k^*+\sum_{i\in\Lambda''\backslash \Lambda'}\varphi_\gamma((B\hat x)_i) >k^*$.  This also contradicts that the optimal function value of problem \eqref{eq:6.1.1} is $k^*$. Therefore, $\|B\hat x\|_0=k^*$. Item (ii) follows immediately.
\end{proof}

\subsection{When $\phi$ is the $\ell_0$ function composed with an affine mapping}
We study in this subsection the case where $\phi$ is the $\ell_0$ function composed with an affine mapping. We show that in this case for any $\lambda>0$, there exists a $\gamma^*\ge 0$ such that both problems \eqref{model:l0} and \eqref{model:3} share the same optimal solution set provided that $\gamma>\gamma^*$.

Let $\phi$  at $x\in\mathbb{R}^n$ be defined as $\phi(x):=\|Ax-b\|_0$,  where $A\in\mathbb{R}^{t\times n}$ and $b\in\mathbb{R}^t$.  In this case, problems \eqref{model:l0} and \eqref{model:3} become
\begin{equation}\label{eq:6.2.1}
\min\{\|Ax-b\|_0+\lambda\|Bx\|_0: x\in\mathbb{R}^n\}
\end{equation}
 and
\begin{equation}\label{eq:6.2.2}
\min\{\|Ax-b\|_0+\lambda\psi_\gamma(Bx): x\in\mathbb{R}^n\}
\end{equation}
 respectively.
From Proposition \ref{prop:l0ALS} and Theorem \ref{thm:SoulExist}, both problems \eqref{model:l0} and \eqref{model:3} have optimal solutions. According to problem \eqref{eq:6.3.1}, for any $p>0$, problem \eqref{eq:6.2.2} can be equivalently rewritten as
\begin{equation}\label{eq:6.2.3}
\min\{\|Ax-b\|_0+\lambda\|v\|_0+\lambda\gamma\|Bx-v\|_p^p: x\in\mathbb{R}^n, v\in\mathbb{R}^m\}.
\end{equation}

We first present below two lemmas needed for the proof of our main result in this subsection.
\begin{lemma}\label{lema:asyLinSetSection6}
Let $C_1\subseteq\mathbb{R}^n$ and $C_2\subseteq\mathbb{R}^m$ be asymptotically linear. Then, $C_1\times C_2$ is also asymptotically linear.
\end{lemma}
\begin{proof}
By the definition of asymptotically linear sets, we require to prove that for each $\rho>0$ and each sequence $\{(x^k, y^k): k\in\mathbb{N}\}$ satisfying
\begin{equation}\label{eq:6.2.4}
(x^k, y^k)\in C_1\times C_2, \|(x^k, y^k)\|_2\rightarrow+\infty, \frac{(x^k, y^k)}{\|(x^k, y^k)\|_2}\rightarrow (\bar x, \bar y),
\end{equation}
there exists $k_0\in\mathbb{N}$ such that
\begin{equation}\label{eq:6.2.5}
(x^k, y^k)-\rho(\bar x, \bar y)\in C_1\times C_2
\end{equation}
 for any $k\ge k_0$.
There are three different cases. The first case is that $\bar x\neq \mathbf{0}_n$ while $\bar y= \mathbf{0}_m$. The second case is that $\bar x= \mathbf{0}_n$  while $\bar y\neq \mathbf{0}_m$. The last case is that $\bar x\neq \mathbf{0}_n$ and $\bar y\neq \mathbf{0}_m$. We next discuss the problem case by case.

For the first case, $\frac{\|y^k\|_2}{\|x^k\|_2}\rightarrow0$ due to $\bar y=\mathbf{0}_m$ and $\frac{(x^k, y^k)}{\|(x^k, y^k)\|_2}\rightarrow (\bar x, \bar y)$. Then $\frac{\|(x^k, y^k)\|_2}{\|x^k\|_2}\rightarrow 1$, therefore, $\frac{x^k}{\|x^k\|_2}\rightarrow\bar x$. Since $C_1$ is asymptotically linear, there exists $k_0\in\mathbb{N}$ such that $x^k-\rho \bar x\in C_1$ for any $k\ge k_0$. Then \eqref{eq:6.2.5} follows  due to $\bar y=\mathbf{0}_m$. By similar method, \eqref{eq:6.2.5} can be obtained for the second case. We finally consider the last case. In this case, $\frac{\|(x^k, y^k)\|_2}{\|x^k\|_2}\rightarrow\frac{1}{\|\bar x\|_2}$ and $\frac{\|(x^k, y^k)\|_2}{\|y^k\|_2}\rightarrow\frac{1}{\|\bar y\|_2}$.  Then \eqref{eq:6.2.4} implies $\frac{x^k}{\|x^k\|_2}\rightarrow \frac{\bar x}{\|\bar x\|_2}$
and $\frac{y^k}{\|y^k\|_2}\rightarrow \frac{\bar y}{\|\bar y\|_2}$.
Because $C_1$ and $C_2$ are asymptotically linear, there exists $k_0\in\mathbb{N}$ such that for any $k\ge k_0$, there holds
\begin{equation}\label{eq:6.2.6}
x^k-\rho \bar x\in C_1, y^k-\rho \bar y\in C_2.
\end{equation}
Then \eqref{eq:6.2.6} implies \eqref{eq:6.2.5} immediately. We complete the proof.
\end{proof}
\begin{lemma}\label{lema:asyLineSetSec6.2}
Let $A\in\mathbb{R}^{t\times n}$, $b\in\mathbb{R}^t$, $\lambda>0$ and $s\ge 0$. Set $O:=\{(x, v)\in\mathbb{R}^n\times\mathbb{R}^m: \|Ax-b\|_0+\lambda\|v\|_0\le s\}$. Then  the set $O$ is either empty or  asymptotically linear.
\end{lemma}
\begin{proof}
We consider the case when $O\neq \emptyset$. In this case, $s>s^*$ with $s^*:=\min\{\|Ax-b\|_0: x\in\mathbb{R}^n\}$. Note that  for any $(x, v)\in\mathbb{R}^n\times\mathbb{R}^m$, $\|Ax-b\|_0+\lambda\|v\|_0\in S$, where $S:=\{i+\lambda j: i\in \mathbb{N}_t^0, j\in\mathbb{N}_m^0\}$. Therefore, the set $O$ can be represented as
$$O=\bigcup_{(i, j)\in\Lambda} O_{ij},
 $$
 where $\Lambda:=\{(i, j): i+\lambda j\le s, i\in\mathbb{N}_t^0, j\in\mathbb{N}_m^0\}$ and
$O_{ij}$ is defined, for any $i\in\mathbb{N}_t^0, j\in\mathbb{N}_m^0$,
as
$$
O_{ij}:=\{(x, v)\in\mathbb{R}^n\times\mathbb{R}^m: \|Ax-b\|_0\le i, \|v\|_0\le j\}.
$$
Since the union of  a finite number of asymptotically linear sets is also asymptotically linear, proving $O$ is asymptotically linear amounts to proving $O_{ij}$ is either empty or asymptotically linear for any  $i\in\mathbb{N}_t^0, j\in\mathbb{N}_m^0$.

Let $i\in\mathbb{N}_t^0, j\in\mathbb{N}_m^0$ and suppose $O_{ij}\neq \emptyset$. Let $O_i:=\{x\in\mathbb{R}^n: \|Ax-b\|_0\le i\}$ and $O_j:=\{v\in\mathbb{R}^m: \|v\|_0\le j\}$. Clearly, $O_{ij}=O_i\times O_j$. By Lemmas \ref{lema:asyLiSetScec5} and \ref{lema:asyLinSetSection6}, we have $O_{ij}$ is asymptotically linear. We then complete the proof.
\end{proof}

Now, we are ready to establish the main result of this subsection in the following theorem.
\begin{theorem}
For any $\lambda>0$, there exits a $\gamma^*\ge 0$ such that both problems \eqref{eq:6.2.1} and \eqref{eq:6.2.2} share the same optimal solution set provided  $\gamma>\gamma^*$.
\end{theorem}
\begin{proof}
According to Proposition \ref{prop:l0ALS} and Theorem \ref{thm:SoulExist}, problems \eqref{eq:6.2.1} and \eqref{eq:6.2.2} have optimal solutions.
Let $\Phi:\mathbb{R}^n\rightarrow \mathbb{R}$ and $H_\gamma: \mathbb{R}^n\times \mathbb{R}^{m}\rightarrow \mathbb{R}$ be the objective function of \eqref{eq:6.2.1} and that of \eqref{eq:6.2.3} respectively.
Let $\Omega$ be the optimal solution set of problem \eqref{eq:6.2.1}. Set $s^*$ and $\tilde \Omega^*$ to be the minimal value and the optimal solution set to problem \eqref{eq:6.2.3}. In order to get this theorem, it amounts to prove
\begin{equation}\label{eq:6.2.8}
\Omega=\{x\in\mathbb{R}^n: (x, v)\in \tilde\Omega^*\} \mathrm{~for~} \gamma>\gamma^*.
\end{equation}

Let $g, h:\mathbb{R}^n\times \mathbb{R}^{m}\rightarrow\mathbb{R}$ defined at $(x, v)\in\mathbb{R}^n\times\mathbb{R}^m$ as $g(x, v):=\|Bx-v\|_p^p$ and $h(x, v):=\|Ax-b\|_0+\lambda\|v\|_0$.
Since $h$ only has finite number of function values, we set $s:=\min\{h(x, v): Bx=v\}$ and
$$\tilde\Omega:=\arg\min\{h(x, v): Bx=v\}.
$$
 Then $\mathbb{R}^n\times \mathbb{R}^{m}$
can be partitioned into two parts: $\mathbb{R}^n\times \mathbb{R}^{m}=O_1\cup O_2$, where
$$O_1:=\{(x, v)\in \mathbb{R}^n\times \mathbb{R}^{m}: h(x, v)\ge s\}$$
 and $O_2:=\mathbb{R}^n\times \mathbb{R}^{m}\backslash O_1$. Clearly, $O_1\neq\emptyset$. Since $h$ is piecewise constant, it is obvious that
 $$O_2=\{(x, v)\in \mathbb{R}^n\times \mathbb{R}^{m}: h(x, v)\le s-\min\{1, \lambda\}\}.$$
 We next consider the minimal value of $H_\gamma$ on $O_1$ and $O_2$ separately.

We first restrict $(x, v)\in O_1$. In this case, the minimal value of $h$ and that of $g$ are $s$ and $0$ respectively. Further, the minimal values of both $h$ and $g$  can be attained at any  $x\in\tilde \Omega$. Therefore, the minimal value of $H_\gamma$ on $O_1$ is $s$.  We next show $\arg\min\{H_\gamma(x, v): (x, v)\in O_1\}=\tilde\Omega$. We only need to prove any $(x^*, v^*)$ being a minimizer of $H_\gamma$ on $O_1$ satisfies $h(x^*, v^*)=s$. If not, that is $h(x^*, v^*)>s$, then $H_\gamma(x^*, v^*)>s$. This contradicts the fact that $s$ is the minimal value of $H_\gamma$ on $O_1$. Therefore, we have $\arg\min\{H_\gamma(x, v): (x, v)\in O_1\}=\tilde\Omega$.

If $O_2=\emptyset$, then the optimal solution set of problem \eqref{eq:6.2.3} is $\tilde\Omega$, that is $\tilde\Omega^*=\tilde\Omega$. It is obvious that \begin{equation}\label{eq:6.2.9}
\Omega=\{x\in\mathbb{R}^n: (x, v)\in \tilde\Omega\},
 \end{equation}
 which implies \eqref{eq:6.2.8}.

We next consider the case when $O_2\neq\emptyset$.
In this case, we set $\rho:=\inf_{O_2} g$. We will show that $\rho>0$. By the previous analysis, $Bx\ne v$ for all $(x, v)\in O_2$, that is $g(x, v)>0$ on $O_2$. Since $g$ is asymptotically level stable by Proposition \ref{prop:leveBoun} and $O_2$ is asymptotically linear from Lemma \ref{lema:asyLineSetSec6.2}, $g$ has global minimizers on $O_2$. Therefore, $\rho>0$. We denote by $s'$  the minimal value  of problem $\min\{H_\gamma(x, v): (x, v)\in O_2\}$. Then  $s'\ge\gamma\lambda \rho$ due to $h\ge 0$.

We have $s^*=\min\{s, s'\}$. Further,  $\tilde \Omega^*=\tilde \Omega$ as $s<s'$. Set $\gamma^*:=s/\lambda\rho$. Then $s<s'$ and  $s^*=s$ when $\gamma>\gamma^*$. Clearly, $\tilde\Omega^*=\tilde\Omega$ whenever $\gamma>\gamma^*$. Therefore, \eqref{eq:6.2.8} follows immediately from \eqref{eq:6.2.9}. We then complete the proof of this theorem.
\end{proof}

\section{Conclusions and extensions}
We investigate in this paper the capped $\ell_p$ approximations with $p>0$ for the composite $\ell_0$ regularization problem. Actually, the capped $\ell_p$ approximation problem \eqref{model:3} can be viewed as a penalty method with the $\ell_p$ penalty function for solving problem \eqref{model:l0}. The existence of optimal solutions to problems \eqref{model:l0} and \eqref{model:3} are established under assumptions that $\phi$ is asymptotically level stable and bounded below. We derive that problem \eqref{model:l0} can be asymptotically approximated by problem \eqref{model:3} as $\gamma$ tends to infinity if $\phi$ is a level bounded function composed with a linear mapping. We further prove that if $\phi$ is the indicator function on an asymptotically linear set or the  $\ell_0$ norm composed with an affine mapping, then problems \eqref{model:l0} and \eqref{model:3} have the same optimal solution set provided that $\gamma>\gamma^*$ for some $\gamma^*>0$.

We emphasize that our analysis in this paper can be extended to investigating capped $\ell_p$ approximations for a more complicated composite $\ell_0+\ell_q$ regularization problem
\begin{equation}\label{eq:l0Gen}
\min\{g(x)+\|Wx\|_q^q+\sum_{s=1}^S \lambda_s\|\|D_sx-d_s\|_2\|_0: x\in\mathbb{R}^n\},
\end{equation}
 where $g:\mathbb{R}^n\rightarrow\bar{\mathbb{R}}$, $W\in\mathbb{R}^{m\times n}$, $q>0$, $D_s\in\mathbb{R}^{m_s\times n}, d_s\in\mathbb{R}^{m_s}$ and $\lambda_s>0$ for $s\in\mathbb{N}_s$. The last term of the objective function is exactly the weighted block (or group) composite $\ell_0$ regularizer. The capped $\ell_p$ approximation for problem \eqref{eq:l0Gen} is
 \begin{equation}\label{eq:lpGer}
 \min\{g(x)+\|Wx\|_q^q+\sum_{s=1}^S \lambda_s\varphi_\gamma(\|D_sx-d_s\|_2): x\in\mathbb{R}^n\},
 \end{equation}
 where $\varphi_\gamma$ is defined by \eqref{def:varphi}. By similar analysis in Section 3 and Section 4, we can obtain that both problems \eqref{eq:l0Gen} and \eqref{eq:lpGer} have optimal solutions if $g$ is asymptotically level stable and bounded below. In addition, problem \eqref{eq:lpGer} asymptotically approximates problem \eqref{eq:l0Gen} as $\gamma$ goes to infinity if $g$ is a level bounded function composed with a linear mapping.

\section{Appendix}
\begin{lemma}\label{lema:varphi}
For $\gamma>0, p>0$,
let $\psi_\gamma:\mathbb{R}^m\rightarrow\mathbb{R} $ at $y\in\mathbb{R}^m$ be defined by $\psi_\gamma(y)=\sum_{i=1}^m\varphi_\gamma(y_i)$, where $\varphi_\gamma$ is defined by \eqref{def:varphi}. Then $\psi_\gamma$ can be written as \eqref{eq:split}.
\end{lemma}
\begin{proof}
It suffices to prove
\begin{equation}\label{eq:app}
\varphi_\gamma(t)=\min\{\|v\|_0+\gamma|t-v|^p: v\in\mathbb{R}\}.
\end{equation}For $t\in\mathbb{R}$, let $\phi_t:\mathbb{R}\rightarrow\mathbb{R}$ defined at $v\in\mathbb{R}$ as $\phi_t(v):=|v-t|^p$. By Corollary \ref{coro:als}, $\phi_t$ satisfies H\ref{hypo:phi}. Let $F_t:(0, +\infty)\rightarrow\mathbb{R}$ defined at $\lambda>0$ as
\begin{equation}\label{eq:app2}
F_t(\lambda):=\min\{\phi_t(v)+\lambda\|v\|_0: v\in\mathbb{R}\}.
\end{equation}
Then proving \eqref{eq:app} amounts to showing
\begin{equation}\label{eq:6.1}
\varphi_\gamma(t)=\gamma F_t(\frac{1}{\gamma}).
\end{equation}
We next dedicate to proving \eqref{eq:6.1}

Since $\phi_t$ satisfies H\ref{hypo:phi} for any $t\in\mathbb{R}$, then Definition \ref{def:sRhoOmega} can be applied to problem  \eqref{eq:app2} by setting $\phi$ to $\phi_t$ and $B$ to the number $1$.
As $t=0$, by Definition \ref{def:sRhoOmega} we obtain $L=0$, $\rho_0=0, s_0=0, \Omega_0=\{0\}$. Then by Theorem \ref{thm:stability} we have
\begin{equation}\label{eq:6.2}
F_0(\lambda)=0.
\end{equation}
As $t\neq 0$, one can get that $L=1$, $\rho_0=0$, $s_0=1$, $\Omega_0=\{t\}$, $\rho_1=|t|^p$, $s_1=0$, $\Omega_1=\{0\}$.
Then by Definition \ref{def:lambda_i} and Theorem \ref{thm:stability}, as $t\neq 0$,
\begin{equation}\label{eq:6.3}
F_t(\lambda)=\begin{cases}
|t|^p, &\lambda>|t|^p,\\
\lambda, &\mathrm{else}.
\end{cases}
\end{equation}
Then, \eqref{eq:6.2} and \eqref{eq:6.3} imply \eqref{eq:6.1}. We complete the proof.
\end{proof}
\bibliographystyle{siam}


\end{document}